\documentclass[10pt, reqno, a4wide]{article}
\usepackage{amssymb, amsmath, amsthm}
\usepackage{ifthen}
\usepackage[all, cmtip]{xy}
\usepackage{pdfpages}
\usepackage{graphics}
\usepackage{mathrsfs}
\usepackage{hyperref}

\newcommand{\Spec}{\mathrm{Spec}}
\newcommand{\red}{\mathrm{red}}
\newcommand{\Zar}{\mathrm{Zar}}

\newcommand{\oxi}{{\overline{\xi}}}

\newtheorem{leer}{}[section]
\newtheorem{thm}[leer]{Theorem}
\newtheorem{conj}[leer]{Conjecture}
\newtheorem{rema}[leer]{Remark} 
\newtheorem{prop}[leer]{Proposition}
\newtheorem{lemm}[leer]{Lemma}
\newtheorem{coro}[leer]{Corollary}

\newtheorem{defi}[leer]{Definition}

\newtheorem{fact}[leer]{Fact}
\newtheorem{setup}[leer]{Setup}
\newtheorem{exam}[leer]{Example}

\newcommand{\ol}[1]{\overline{#1}}

\newcommand{\plim}{\mathop{\varprojlim}\limits}

\newcommand{\XX}{{\mathcal X}}
\newcommand{\YY}{{\mathcal Y}}
\newcommand{\ZZ}{{\mathcal Z}}

\newcommand{\AAA}{\mathscr{A}}

\newcommand{\OOO}{\mathscr{O}}

\newcommand{\Cc}{{\mathbb{C}}}

\newcommand{\Nn}{{\mathbb{N}}}

\newcommand{\Pp}{{\mathbb{P}}}
\newcommand{\Qq}{{\mathbb{Q}}}
\newcommand{\Rr}{{\mathbb{R}}}

\newcommand{\Zz}{{\mathbb{Z}}}

\newcommand{\oa}{{\ol{a}}}
\newcommand{\ob}{{\ol{b}}}

\newcommand{\of}{{\ol{f}}}

\newcommand{\os}{{\ol{s}}}

\newcommand{\oy}{{\ol{y}}}
\newcommand{\oz}{{\ol{z}}}

\newcommand{\oE}{{\ol{E}}}
\newcommand{\oF}{{\ol{F}}}

\newcommand{\oK}{{\ol{K}}}

\newcommand{\oY}{{\ol{Y}}}

\newcommand{\Mor}{\mathrm{Mor}}

\newcommand{\Gal}{\mathrm{Gal}}

\newcommand{\chara}{\mathrm{char}}

\begin{document}

\title{Fibration theorems for varieties with the weak Hilbert property
}
\author{Sebastian Petersen}
\maketitle

\begin{abstract}
The weak Hilbert property (WHP) for varieties over fields of characteristic zero was introduced by Corvaja and Zannier in 2017. There exist integral variants of WHP for arithmetic schemes. 
We present new fibration theorems for both the WHP and its integral analogue. Our primary fibration result, in a sense dual to the mixed fibration theorems of Javanpeykar and Luger, establishes for a smooth proper morphism $f: Y \to Z$ of smooth connected varieties, that if $Z$ has the strong Hilbert property (HP) and the generic fiber has WHP, then the total space $Y$ also has WHP. As an application, we use this result in combination with previous work by Corvaja, Demeio, Javanpeykar, Lombardo, and Zannier and in combination with recent work of Javanpeykar to show that certain non-constant abelian schemes over HP varieties possess WHP. For integral WHP, we prove a new fibration theorem for proper smooth morphisms with a section, which generalizes earlier product theorems of Javanpeykar and Wittenberg, and of Luger. A key lemma gives information about the structure of covers of $Y$ whose branch locus is not dominant over $Z$.
\end{abstract}

\tableofcontents

\section{Introduction}
\subsection{Previous work}

The concept of a variety with the Hilbert property HP was introduced by Colliot-Th\'el\`ene--Sansuc 
and Serre (cf. \cite{CS}, \cite[3.1.1, 3.1.2]{SerTopics}). It turns out, however, that over a number field the smooth proper connected varieties with HP form a quite restricted class of varieties even among those having a Zariski dense set of rational points: Corvaja and Zannier proved that for every smooth proper connected variety $Y$ over a number field $K$, if $Y$ has HP, then the geometric \'etale fundamental group 
$\pi_1(Y_\oK)$ is trivial (cf. \cite[Thm. 1.6]{CZ}).  For example a non-trivial abelian variety over a number field never has HP. This led Corvaja and Zannier to the concept of a variety with the weak Hilbert property WHP (cf. \cite[§2]{CZ}, \cite[Definition 1.2]{CDJLZ}). 

\begin{defi} Let $Y$ be a normal connected variety over a field $K$ of characteristic zero. A subset $M$ of $Y(K)$ is said to be {\bf thin} (resp. {\bf strongly thin}) in $Y$ if there exists a proper closed subset $C$ of $Y$, a finite set $J$ and for each $j\in J$ a normal connected variety $X_j/K$ and a finite surjective $K$-morphism $g_j: X_j\to Y$ such that 
$\deg(g_j)\ge 2$ for all $j\in J$ (resp. such that $g_j$ is ramified for all $j\in J$) and such that $M\subset C(K)\cup \bigcup_{j\in J} X_j(K)$. 
A $K$-variety $Y$ has {\bf HP (resp. WHP) over $K$} if $Y$ is normal and connected and $Y(K)$ is not thin (resp. not strongly thin) in $Y$. 
We say that $Y$ has {\bf HP (resp. WHP) over $K$ potentially} if there exists a finite extension $K'/K$ such that $Y_{K'}$ has HP (resp. WHP) over $K'$. 
\end{defi}

This definition of HP / WHP is equivalent to the classical one
(cf. \cite[p. 189]{CS}, \cite[3.1.1, 3.1.2]{SerTopics}, \cite{BFP3}) except that some authors do not build a normality assumption into HP. 
The following conjecture is implied by a central conjecture of Corvaja and Zannier 
(cf. \cite[Section 2, p. 9, Question-Conjecture 2]{CZ}, \cite[Conjecture 1.9]{CDJLZ}). 

\begin{conj} \label{conj:CZ} Let $K$ be a finitely generated field of characteristic zero. 
Let $Y$ be a smooth connected $K$-variety.
If $Y(K)$ is Zariski dense in $Y$, then $Y$ has WHP potentially over $K$. 
\end{conj}

A recent breakthrough was the paper \cite{CDJLZ} where Corvaja, Demeio, Javanpeykar, Lombardo and Zannier proved Conjecture
\ref{conj:CZ} for  abelian varieties and the paper \cite{Lug3} where Luger, building on \cite{CDJLZ}, proved this conjecture for connected algebraic groups. The recent survey \cite{FeJav} of Fehm and Javanpeykar gives a good picture about other cases where this conjecture is known. Let us consider the following conjecture. 

\begin{conj} \label{conj:fib}
Let $K$ be a finitely generated field of characteristic zero. Let $Y$ and $Z$ be smooth connected $K$-varieties and 
$f: Y\to Z$ a 
surjective morphism. Assume that $Z$ has WHP over $K$. 
\begin{enumerate}
\item[(a)] Let $\Xi$ be the set of all $z\in Z(K)$ such that the fibre 
$Y_z=Y\times_{Z,f} \Spec(k(z))$ has 
WHP over $K$. If $\Xi$ is not strongly thin in $Z$, then $Y$ has WHP over $K$. 
\item[(b)] If the generic fibre of $f$ has WHP over the function field $R(Z)$ of $Z$, then $Y$ has WHP over $K$. 
\end{enumerate}
\end{conj}

Conjecture \ref{conj:fib} is quite natural: In the situation of Conjecture \ref{conj:fib} it is easy to see that $Y(K)$ is Zariski dense in $Y$ and thus Conjecture \ref{conj:CZ} predicts that $Y$ ``should have'' WHP over $K$ at least potentially. Conjecture \ref{conj:fib} is open, but it is known in some cases. We usually refer to results towards Conjecture \ref{conj:fib} as ``fibration theorems''.
An early theorem of that kind was established by Bary-Soroker, Fehm and the author in 2015 (cf. \cite{BFP1}). 
If $Z$ has HP and the set $\{z\in Z(K): \mbox{$Y_z$ has HP}\}$ is not thin in $Z$, then $Y$ is known to have HP by 
\cite[Thm. 1.1.]{BFP1}. 
The following variant of the above fibration theorem was established recently by the same authors: If $Z$ has HP over $K$ and the generic fibre $Y_{R(Z)}$ of $f$ has HP over $R(Z)$, then $Y$ has HP over $K$ by \cite[Prop. 3.4]{BFP3}. 
The following product theorem can be obtained from \cite[Thm. 1.1.]{BFP1} or from \cite[Prop. 3.4]{BFP3}: If $Z_1$ and $Z_2$ are $K$-varieties that both have HP, then the $K$-variety $Z_1\times_K Z_2$ has HP (cf. \cite[Cor. 3.4]{BFP1}, \cite[Cor. 3.5]{BFP3}). 
There is the so-called ``mixed fibration theorem'' of Luger \cite{Lug4}
which itself generalizes a fibration theorem of Javanpeykar \cite[Thm. 1.3]{Jav}: If $Z$ has WHP and 
$\{z\in Z(K): \mbox{$Y_z$ has HP}\}$ is not strongly thin in $Z$, then the mixed fibration theorem implies that $Y$ has WHP. 
Furthermore there is the following product theorem of Javanpeykar and Wittenberg that appeared within the seminal paper \cite{CDJLZ}: 
If $Z_1$ and $Z_2$ are smooth proper $K$-varieties that both have WHP, then the $K$-variety $Z_1\times_K Z_2$ has WHP. 
by \cite[Thm 1.9]{CDJLZ}. 

There are various variants of HP and WHP in the literature addressing integral points on arithmetic schemes rather than rational points on $K$-varieties. The following definition is inspired by Luger's work  \cite{Lug1},  \cite{Lug2},  \cite{Lug3}. In our notation a variety over a scheme $S$ is a $S$-scheme that is separated and of finite type over $S$ (cf. page \pageref{not}). 

\begin{defi} Let $S$ be a noetherian connected regular scheme whose function field $K=R(S)$ is of characteristic zero. Let $\XX$ be a variety over $S$ and $X=\XX_K$. 
We denote by $\XX(S)^{(1)}\subset X(K)$ the set of all near $S$-integral points of $X$ (cf. Definition \ref{defi:near}). 
We say that $\XX$ {\bf has HP (resp. WHP) integrally over $S$} if the $K$-variety $X$ is normal and connected and $\XX(S)^{(1)}$ is not thin (resp. not strongly thin) in $X$. 
We say that $\XX$ has {\bf WHP over $S$ potentially} if there exists a 
noetherian connected regular scheme $T$ and a quasi-finite dominant morphism $T\to S$ 
such that $\XX\times_S T$ has WHP integrally over $T$. 
\end{defi}

The following conjecture is motivated by a conjecture of Corvaja and Zannier (cf.  \cite[p.11]{CZ}). It 
agrees essentially with \cite[Conjecture 1.5]{Lug3}. See also \cite[Conjecture 1.12]{BJL}. 
\begin{conj} Let $S$ be a regular connected dominant $\Zz$-variety. Let $\YY$ be a smooth $S$-variety and assume that $Y:=\YY_K$ is connected. If $\YY(S)^{(1)}$ is Zariski dense in $Y$, then $\YY$ has WHP over $S$ potentially. 
\end{conj}

Luger has product theorems in the context of arithmetic schemes; we refer to 
\cite{Lug1} and \cite{Lug3} for this interesting recent development. 
Let $S$ be a regular connected dominant $\Zz$-variety. Let $\XX, \YY$ be $S$-varieties. If $\XX$ and $\YY$ both have HP (resp. WHP) integrally over $S$, then
$\XX\times_S \YY$ has HP (resp. WHP) integrally over $S$ by a recent theorem of Luger \cite[Thm. 1.5]{Lug1} (resp. \cite[Thm. 1.4]{Lug3}). 
\cite[Thm. 1.4]{Lug3} generalizes \cite[Thm. 1.9]{CDJLZ}. 

The above fibration theorems are usually at game, when results towards Conjecture \ref{conj:CZ} are established.

\subsection{Main results}

The following theorem is the first main result of our paper. 
It is in a sense dual to the above mentioned mixed fibration theorems of Javanpeykar \cite[Thm. 1.3]{Jav} and Luger \cite{Lug4}. 

\begin{thm}\label{thm:1} (cf. Corollary \ref{coro:main:smoothproper}) 
Let $K$ be a field of characteristic zero. Let $Y$ and $Z$ be connected smooth $K$-varieties. 
Let $f: Y\to Z$ be a smooth proper morphism. 
Assume that $Z$ has HP over $K$ and that the generic fibre $Y_{R(Z)}$ has WHP over the function field $R(Z)$ of $Z$. Then 
$Y$ has WHP over $K$. 
\end{thm}

 The properness assumption in Theorem \ref{thm:1} can be weakened to some extent. We refer to Corollary \ref{coro:pi2} of the main text for such a generalization. The following remark explains the merit and limitation of Theorem \ref{thm:1}. 
 
 \begin{rema} \label{rema:crit} Let $K$ be a field of characteristic zero. Let $Y$ and $Z$ be connected smooth $K$-varieties. 
Let $f: Y\to Z$ be a dominant morphism. 
Assume that $Z$ has HP over $K$ and that the generic fibre $Y_{R(Z)}$ is a smooth and proper $R(Z)$-variety that satisfies WHP over $R(Z)$.
By the spreading-out principles of EGA4 there exists a non-empty open subscheme $U$ of $Z$ such that $f_U: Y_U\to Z_U$ 
is smooth and proper. Thus the ``stripe'' $Y_U$ has WHP over $K$ by Theorem \ref{thm:1}. 
Nevertheless the following Conjecture \ref{conj:ell2}, which is implied by Conjecture \ref{conj:fib} in case $K/\Qq$ is finitely generated, remains open at present unless $f$ is smooth and proper. Section \ref{sec:open} explains - in a relevant special case - where the problem lies. 
 \end{rema}
 
 \begin{conj} \label{conj:ell2} 
 In the situation of Remark \ref{rema:crit} $Y$ has WHP over $K$. 
\end{conj}

  We combine Theorem \ref{thm:1} with the main result of \cite[Thm. 1.3]{CDJLZ} to prove Conjecture \ref{conj:CZ} in the case of 
certain abelian schemes over HP varieties.

\begin{coro}  (cf. Corollary \ref{coro:as})  \label{cor:2} Let $K$ be a finitely generated 
field of characteristic zero. Let $Z$ be a connected smooth $K$-variety and $A/Z$ an abelian scheme. 
Assume that $A(R(Z))$ is Zariski dense in $A_{R(Z)}$ and that $Z$ has HP over $K$. Then $A$ has WHP over $K$. 
\end{coro}

 If there exists an abelian variety $A_0/Z$ and an isomorphism $A_0\times_K Z\cong A$, then alternatively \cite[Thm. 1.3]{CDJLZ} and the mixed fibration theorem from \cite{Lug4}
 imply that $A$ has WHP over $K$. Thus the main merit of Corollary \ref{cor:2} is its ability to treat certain non-constant abelian schemes. 
 
 \begin{exam} Let $K$ be a finitely generated field of characteristic zero and $Z$ a non-empty open subscheme of $\Pp^1_K$. Then $Z$ has HP over $K$ by the classical Hilbert irreducibility theorem. Let $A/Z$ be an abelian scheme and assume that $A(R(Z))$ is Zariski dense in $A_{R(Z)}$. Then $A$ has WHP over $K$ by Corollary \ref{cor:2}.
 \end{exam}




A recent theorem of Javanpeykar (cf. \cite[Theorem B]{Jav2}), building on work \cite{XiYu} of Xie and Yuan, 
establishes WHP for certain traceless abelian varieties over function fields of characteristic zero. Combining Javanpeykar's theorem with Theorem \ref{thm:1} we obtain the following Corollary. 

\begin{coro} (cf. Corollary \ref{coro:as2}) \label{cor:3}
Let $K$ be a field of characteristic zero. Let $Z$ be a connected smooth $K$-variety, $F=R(Z)$ the function field of $Z$  and $A/Z$ an abelian scheme. 
Assume that $A(F)$ is Zariski dense in $A_{F}$
and that $Z$ has HP over $K$. If the Chow trace $\mathrm{Tr}_{\oF/\oK}(A_\oF)$ is zero, then $A$ has WHP over $K$. 
\end{coro}

Furthermore we establish the following fibration theorem concerning the  integral weak Hilbert property. It generalizes \cite[Thm. 1.4]{Lug3} and
 \cite[Thm. 1.9]{CDJLZ}.

\begin{thm} (cf. Theorem \ref{thm:iwhp}, Remark \ref{rema:iwhp}) \label{thm:2} 
Let $S$ be a regular dominant $\Zz$-variety and $K=R(S)$ its function field. Let $\YY$ and $\ZZ$ be $S$-varieties, $F: \YY\to \ZZ$ an $S$-morphism, $Y=\YY_K$, $Z=\ZZ_K$ and $f=F_K$. Assume that $Y$ and $Z$ are normal connected $K$-varieties. 
Assume that $f$ is proper, smooth and has a section\footnote{Then all fibres of $f$ are geometrically connected, cf. 
Corollary \ref{coro:conn}.}. Let $\Xi_0$ be the set of all $z\in \ZZ(S)^{(1)}$ such that $Y_z$ has WHP over $K$. If $\Xi_0$ is not strongly thin in $Y$, then $\YY$ has WHP integrally over $S$. 
\end{thm}

In the main text Theorem \ref{thm:iwhp} offers a strengthening of Theorem \ref{thm:2} in terms of strongly thin sets that was suggested by Ariyan Javanpeykar. Also the properness assumption in Theorem \ref{thm:2} can be weakened to some extent (cf. Theorem \ref{thm:iwhp}, Corollary 
\ref{coro:iwhp-setb}). 
Moreover the assumptions on $S$ can be relaxed considerably; it suffices to assume that $S$ is a Hermite-Minkowski scheme in the sense of Definition \ref{defi:HM} below. 

\subsection{Strategy and outline of the paper}

The proofs of Theorem \ref{thm:1} and Theorem \ref{thm:2} rely heavily on the following Lemma about the structure of so-called vertical ramified covers. 

\begin{lemm} \label{lemm:key} Let $K$ be a field of characteristic zero and $X, Y, Z$ normal connected $K$-varieties. Let $\xi$ (resp. $\oxi$) be the generic (resp. a geometric generic) point of $Z$. 
Let $f: Y\to Z$ be a smooth morphism and assume that the generic fibre $Y_\xi$ is geometrically connected over $R(Z)$. Let
$g: X\to Y$ a finite surjective ramified morphism. Assume that $g$ is vertical over $f$ in the sense that $g_\xi: X_\xi\to Y_\xi$ is unramified.\footnote{i.e. such that the branch locus of $g$ is not dominant over $f$} Assume that the natural map $\pi_1(Y_\oxi)\to \pi_1(Y)$ is injective. Then there exist normal $K$-varieties $X''$ and $Z''$, a finite \'etale morphism $X''\to X$, a finite surjective ramified morphism $Z''\to Z$ and a $Z$-morphism $X''\to Z''$. 
\end{lemm}

Note that in the special case where $Y=Z_1\times_K Z_2$, $Z=Z_1$ and $f: Y\to Z$ is just the projection of the product, the structure of vertical ramified covers was well-understood in \cite[Lemma 2.1]{Lug3} and \cite[Lemma 2.17]{CDJLZ}. Lemma \ref{lemm:key} provides similar information in the more general situation and is the technical heart of our paper. 

We can now briefly describe the strategy in the proof of Theorem \ref{thm:1}, which is to some extent inspired by \cite[Proposition 2.4]{BFP2}. Let $K, Y, Z, f$ be as in Theorem \ref{thm:1} and let $(X_i)_{i\in I}$ be a finite family of normal connected $K$-varieties and $(g_i: X_i\to Y)_{i\in I}$ a finite family of 
finite surjective ramified morphisms. Let $V\subset Y$ be a non-empty open subset. We have to prove
that there exists a point $v\in V(K)$ such that $v\notin \bigcup_{i\in I} g_i(X_i(K))$. Let $I^v=\{i\in I| g_i \ \mbox{is vertical over $f$}\}$ and $I^h=I\setminus I^v$. As $Y_\xi$ has WHP over $R(Z)$ there exists a point $\sigma\in V_\xi(R(Z))$ with $\sigma\notin \bigcup_{i\in I^h} g_{i, \xi}(X_{i,\xi}(R(Z)))$, and we can extend $\sigma$ to a section $s: D\to Y$ of $f$ defined over an open subscheme $D\subset Z$ that contains all codimension $1$ points of $Z$. Let $T_i=\{z\in D(K): s(z)\in g_i(X_i(K))\}$. It is not difficult to show that $T_i$ is thin for every $i\in I^h$ (cf. Proposition \ref{prop:sec:good:pts} (b)). In the situation under consideration the natural map $\pi_1(Y_\oxi)\to \pi_1(Y)$ can be shown to be 
injective (cf. Fact \ref{fact:SGA1ht}, Corollary \ref{coro:SGA1ht}). Hence one can apply Lemma \ref{lemm:key} to the covers $g_i$ with $i\in I^v$. With the help of the information from
Lemma \ref{lemm:key} one can then prove that 
$T_i$ is thin for every $i\in I^v$ (cf. Proposition \ref{prop:sec:good:pts} (a)). Thus $T=\bigcup_{i\in I} T_i$ is thin. As $Z$ has HP one can choose $z\in D(K)\setminus T$. Then $v=s(z)$ has the desired properties. 

Our proof of Theorem \ref{thm:2} roughly follows \cite[Section 3.4-3.5]{CDJLZ} and \cite[Section 2-4]{Lug3}, 
but using the Lemma \ref{lemm:key} instead of 
\cite[Lemma 2.1]{Lug3} and \cite[Lemma 2.17]{CDJLZ}, and with a certain simplification made possible by invoking 
in the proof 
of Lemma \ref{lemm:weil} the recent \cite[Lemma 2.1]{BFP3}, which was not available when \cite{Lug3} and \cite{CDJLZ} where written. 

Sections 2 contains background material on relative normalization and on near integral points. In Section 3 we study several types of covers.
In Section 4 we discuss to some in which circumstances, in a situation as above, $\pi_1(Y_\oxi)\to \pi_1(Y)$ is known to be injective. More on that can be found in Section 9. 
In Section 5 the above lemma about the structure of vertical ramified covers is established. In Section 6 we study the pullback of vertical ramified covers along sections with the help of the information coming from that lemma. Sections 7 and 8 then contain the proof of Theorem \ref{thm:1} and Theorem \ref{thm:2}, respectively, the proofs of the above corollaries, and some complements. In Section 9 we invoke some higher  \'etale homotopy theory in order to weaken the properness assumption in Theorem \ref{thm:1} and Theorem \ref{thm:2} to some extent. Section 10 contains complements and a discussion an open problem. 

{\bf Acknowledgements}

I thank Jakob Stix for his invaluable guidance with higher \'etale homotopy theory, and more. I am particularly thankful for his permission to reproduce the proofs of Theorem \ref{thm:stix}, Proposition \ref{prop:ascheme}, and Proposition \ref{prop:htgr}, which I learned from him. I am grateful to  Ariyan Javanpeykar for extensive discussions and thoughtful comments, and to Arno Fehm, Wojciech Gajda and Olivier Wittenberg for their insightful feedback on earlier versions of this manuscript.

\section*{Notation}\label{not}
For a point $x\in X$ we denote by $k(x)$ the residue field. 
A {\bf codimension $d$ point} of $X$ is a point $x\in X$ such that the local ring  $\OOO_{X,x}$ 
has Krull dimension $d$. 
A {\bf geometric point} of a scheme $X$ is a morphism $x: \Spec(\Omega)\to X$ where $\Omega$ is an algebraically closed field. We then put $k(x):=\Omega$. The function field of an integral scheme $X$ is denoted by $R(X)$. 
A {\bf geometric generic point} of $X$ is then a geometric point $x: \Spec(\Omega)\to X$ (with $\Omega$ algebraically closed) localized at the generic point of $X$, i.e. such that there exists a $X$-morphism $\Spec(\Omega)\to \Spec(R(X))$. 
A morphism of schemes is said to be a {\bf cover} if it is finite and surjective. 
Let $S$ be a scheme. A {\bf variety over $S$} is an $S$-scheme that is separated and of finite type over $S$.  Following \cite[6.8.1]{EGAIV2} a morphism of finite type $f: X\to Y$ is said to be a {\bf normal morphism} if $f$ is flat and all fibres of $f$ are geometrically normal. 
For a scheme $X$ and a geometric point $x$ of $X$ we denote by $\pi_1(X,x)$ the {\bf \'etale fundamental group} in the sense of SGA1. It is a profinite group. In situations where the base point is irrelevant we write $\pi_1(X)$ instead of $\pi_1(X,x)$. 
The empty scheme $\emptyset$ is neither connected nor irreducible in my notation.

\section{Preliminaries}

\subsection{Preliminaries on relative normalization}

We shall use without further mention the facts summarized in the following Remark. 

\begin{rema} \label{rema:basic} If in a noetherian connected scheme all local rings are domains, then it is integral (cf. \cite[6.1.10]{EGAI}). 
By definition, a scheme is normal if all its local rings are normal domains. Thus every noetherian normal connected scheme is integral. 
Every normal scheme is unibranch.
If $X$ is an integral scheme, $Y$ an unibranch integral scheme  and  $f: X\to Y$ is a finite unramified morphism, then $f$ is automatically \'etale (cf. \cite[18.10.1]{EGAIV4}). A finite (resp. \'etale) morphism is closed (resp. open). If $f: X\to Y$ is a finite \'etale morphism and $Y$ connected, then $f$ is surjective. 
If $f: X\to Y$ is a dominant morphism of integral schemes (resp. normal integral schemes), then the generic fibre is integral (resp. normal and integral). 
A normal variety over a field $K$ of characteristic zero is automatically geometrically normal (cf. \cite[6.14.2]{EGAIV4}).
If $u: X\to Y$ and $v: Y\to Z$ are finite surjective morphisms of normal connected $K$-varieties, then $v\circ u$ is unramified if and only if $u$ and $v$ are both unramified (cf. \cite[Lemma 2.4(d)]{BFP2}). Every scheme of finite type over a field is a Nagata scheme (cf. \cite[\href{https://stacks.math.columbia.edu/tag/032E}{Tag 032E}, Lemma 10.162.16]{stacks-project},
\cite[\href{https://stacks.math.columbia.edu/tag/033R}{Tag 033R, Lemma 28.13.1}]{stacks-project})
\end{rema}

Let $K$ be a field. Let $Y, Z$ be $K$-varieties and $f: Y\to Z$ a morphism. Let $\AAA$ be the integral closure of $\OOO_Z$ in $f_*(\OOO_Y)$ (cf. \cite[\href{https://stacks.math.columbia.edu/tag/0BAK}{Tag 0BAK}, Definition 29.53.2]{stacks-project}). 
The $Z$-scheme $Z':=\Spec(\AAA)$ is called the {\bf relative normalization} of $Z$ in $Y$. It comes equipped with a $Z$-morphism $f': Y\to Z'$ that is induced by 
$\AAA\to f_*\OOO_Y$ under the adjunction $\Mor_Z(Y, \Spec(\AAA))\cong \mathrm{Alg}_{\OOO_Z}(\AAA, f_* \OOO_Y)$. Let $u: Z'\to Z$ be
the structure morphism of the $Z$-scheme $Z'$. Then
\begin{align}\label{eq:relnorm}
Y\buildrel f'\over \longrightarrow Z'\buildrel u\over \longrightarrow Z
\end{align}
is a factorization of $f$, i.e. $f=u\circ f'$. 

\begin{fact}\label{fact:relnorm} (cf. \cite[\href{https://stacks.math.columbia.edu/tag/0BAK}{Tag 0BAK}]{stacks-project})
\begin{enumerate}
\item[(a)] The morphism $u$ is finite (cf. \cite[\href{https://stacks.math.columbia.edu/tag/0BAK}{Tag 0BAK}, Lemma 29.53.14]{stacks-project}).
\item[(b)] The morphism $f'$ induces a bijection of connected components. If $Y$ is normal (resp. normal and connected), then $Z'$ is normal (resp. normal and connected) and $f'$ is dominant. 
(cf. \cite[\href{https://stacks.math.columbia.edu/tag/0BAK}{Tag 0BAK}, Lemma 29.53.13]{stacks-project}).
\item[(c)] For every $Z$-scheme $Z''$ which is finite over $Z$ the map $$\Mor_Z(Z', Z'')\to \Mor_Z(Y, Z''),\ h\mapsto h\circ f'$$
is bijective. 
\end{enumerate}
\end{fact}

\begin{fact}\label{fact:stein} If $f$ is proper, then $f'$ is proper and all fibres of $f'$ are geometrically connected and  the factorization
\eqref{eq:relnorm} agrees with the Stein factorization of $f$ (cf. \cite[\href{https://stacks.math.columbia.edu/tag/03GX}{Tag 03GX, Theorem 37.53.4}]{stacks-project} ) 
\end{fact}

We shall make frequent use of the following proposition.

\begin{prop} \label{prop:stein-np}  Let $K$ be a field of characteristic zero.
Let $Y, Z$ be connected normal $K$-varieties. 
Let $f: Y\to Z$ be a dominant $K$-morphism. Let $Z'$ be the relative normalization of $Z$ in $Y$ and consider the canonical factorization
$Y\buildrel f'\over \longrightarrow Z'\buildrel u\over \longrightarrow Z$ of $f$, as above. There exists a dense open subscheme 
$V$ of $Z$ with the following properties. 
\begin{enumerate}
\item[(a)] $u_V: Z'_V\to V$ is finite \'etale. 
\item[(b)] The fibres of $f'_V: Y_V\to Z'_V$ are normal and geometrically connected.
\item[(c)] The fibres of $f_V: Y_V\to V$ are normal and not empty. 
\item[(d)] If the generic fibre of $f$ is geometrically connected, then $u$ is an isomorphism and the fibres of $f_V: Y_V\to V$ are normal and geometrically connected. 
\end{enumerate}
\end{prop}

\begin{proof} We know that $u$ is finite and $Z'$ normal by Fact \ref{fact:relnorm}. As $f$ is dominant it follows that $u$ is also dominant. Thus $u$ must be surjective. 
The generic fibre of $u$ is $\Spec(R(Z'))$ and $R(Z')/R(Z)$ is separable because $\mathrm{char}(R(Z))=0$. 
By \cite[17.7.11]{EGAIV4} there exists a dense open subscheme $V$ of $Z$ such that $u_V: Z'_V\to V$ is unramified, and then (a) follows because $Z$ and $Z'$ are normal. 

Let $L$ be the integral closure of $R(Z)$ in $R(Y)$. 
We claim that $Z'$ is the normalization of $Z$ in $L$ and $R(Z')=L$.  
For the proof of that claim we can assume that $Z=\Spec(A)$ is affine with a normal domain $A$.
By construction $Z'=\Spec(B)$ where $B$ is the integral closure of $A$ in the 
normal domain $\Gamma(Y, \OOO_Y)$ and $A\to \Gamma(Y, \OOO_Y)$ is injective because
$f$ is dominant. 
Thus $B$ is also the integral closure  of $A$ in $R(Y)$. As $u$ is finite we have $[R(Z'): R(Z)]<\infty$, hence $R(Z')\subset L$. 
It follows that $B$ is the integral closure of $A$ in $L$. In particular $R(Z')=\mathrm{Quot}(B)=L$. 
This finishes up the proof of the claim. 

The $K$-varieties $Y$ and $Z'$ are normal and connected and $f'$ is dominant. The generic fibre $F$ of $f'$ is a normal and integral $R(Z')$-variety. As $\chara(R(Z))=0$ it is geometrically normal. 
The function field of $F$ is $R(Y)$. By the claim
$R(Z')$ is algebraically closed in $R(Y)$. Thus $F$ is geometrically irreducible over $R(Z')$ by \cite[4.5.9(c)]{EGAIV2}. 
Let $M$ be the set of all $x\in X$ such that $f'^{-1}(f'(x))$ is a geometrically normal and geometrically connected $k(f(x))$-scheme. Then $f'(M)$ is a constructible subset of $Z'$ (cf. \cite[9.7.7]{EGAIV3}, \cite[9.9.4]{EGAIV3}, \cite[1.8.5]{EGAIV1}) 
and contains the generic point of $Z'$.  Thus $f'(M)$ contains a dense open subset $W$ of $Z'$, and 
after replacing $V$ by $V\setminus u(Z'\setminus W)$ (b) holds true. Part (c) is immediate from (a) and (b). 
For the proof of (d) assume that the
generic fibre of $f$ is geometrically connected. Then $[R(Z'):R(Z)]=1$. 
Thus $u$ is an isomorphism and by (b) the fibres of $f_V: X_V\to Y_V$ are normal and geometrically connected. 
\end{proof}

\begin{coro} \label{coro:conn} Let $Y$ and $Z$ be connected normal varieties over a field $K$ of characteristic zero. Let $f: Y\to Z$  be a smooth proper
morphism. If $f$ has a section, then $Y(R(Z))\neq \emptyset$. If $Y(R(Z))\neq \emptyset$, then the generic fibre of $f$ is geometrically connected. 
If the generic fibre of $f$ is geometrically connected, then all fibres of $f$ are geometrically connected. 
\end{coro}

\begin{proof} Let $\xi$ be the generic point of $Z$. If $f$ has a section, then this section accounts for an $R(Z)$-rational point of $Y_\xi$, i.e. $Y(R(Z))\neq \emptyset$. The generic fibre $Y_\xi$ is smooth and integral. If $Y(R(Z))\neq\emptyset$, then it follows that $Y_\xi$ has an $R(Z)$-rational point and thus $Y_\xi$ is then geometrically connected as $R(Z)$-scheme. Assume from now on that $Y_\xi$ is geometrically connected as $R(Z)$-scheme.
Let $Z'=\Spec(f_*\OOO_Y)$ and consider the Stein factorization $Y\buildrel f'\over \longrightarrow
Z'\buildrel u\over \longrightarrow Z$ of $f$. Then all fibres of $f'$ are geometrically connected (cf. Fact \ref{fact:stein})), $Z'$ is a normal connected
$K$-variety and $u$ a finite surjective morphism (cf. Fact \ref{fact:relnorm}) of normal connected $K$-varieties. The morphism $f_\xi: Y_\xi\to \Spec(R(Z))$ factors as
$Y_\xi\to \Spec(R(Z'))\to \Spec(R(Z))$. As $Y_\xi$ is geometrically connected as $R(Z)$-scheme we have $R(Z')=R(Z)$. 
The morphism $u$ identifies $Z'$ as the normalization of $Z$ in $R(Z')$. It follows that $u$ is an isomorphism. Thus all fibres of $f$ are geometrically connected. 
\end{proof}

Let $K$ be a field and $Z$ a $K$-variety. We adopt the following definition from the stacks project.

\begin{defi} (cf. \cite[\href{https://stacks.math.columbia.edu/tag/035E}{Tag 035E, Definition 29.54.1}]{stacks-project})
The {\bf normalization $\nu: Z^\nu\to Z$ of $Z$} is the relative normalization of $Z$ in the $Z$-scheme $\coprod_{\eta} \Spec(k(\eta))$ where
$\eta$ runs through the generic points of the (finitely many) irreducible components of $Z$. 
\end{defi}

\begin{fact}\label{fact:norm} 
\begin{enumerate}
\item[(a)] The morphism $\nu: Z^\nu\to Z$ is finite and $Z^\nu$ is a normal $K$-variety (cf. Fact \ref{fact:relnorm}).
\item[(b)] The morphism $\nu: Z^\nu\to Z$ factors as $Z^\nu\to Z_\red\to Z$ and $Z^\nu\to Z_\red$ is the relative normalization of $Z_\red$
in $\coprod_{\eta} \Spec(k(\eta))$ (cf. \cite[\href{https://stacks.math.columbia.edu/tag/035E}{Tag 035E, Lemma 29.54.2}]{stacks-project}). 
\item[(c)] There is a $Z$-isomorphism $Z^\nu\cong \coprod_I I^\nu$ where $I$ runs through the irreducible 
components of $Z$ and $I^\nu$ is the normalization of $I$ (cf. \cite[\href{https://stacks.math.columbia.edu/tag/035E}{Tag 035E, 29.54.6}]{stacks-project})
\end{enumerate}
\end{fact}

\begin{lemm} \label{lemm:disconcover} Let $K$ be a field of characteristic zero and $Z$ a normal connected $K$-variety. 
Let $F$ be a (possibly non-normal, non-connected) $K$-variety and $\gamma: F\to Z$ a finite surjective morphism. 
Let $\nu: F^\nu\to F$ be the normalization of $F$.
Then $\gamma(F(K))$ is thin (resp. strongly thin) in $Z$ if and only if $\gamma\circ \nu(F^\nu(K))$ is 
thin (resp. strongly thin) in $Z$. 
\end{lemm}

\begin{proof} Let $\xi$ be the generic point of $Z$. The morphism $\nu: F^\nu\to F$ factors as 
$F^\nu\buildrel{\nu'}\over \longrightarrow F_\red\buildrel i\over \longrightarrow F$ by Fact \ref{fact:norm}, and $F^\nu\to F_\red$ is the normalization of $F_\red$. The morphism $i: F_\red(K)\to F(K)$ is bijective. Hence $\gamma(F(K))=\gamma\circ i(F_\red(K))$. Hence we can assume that $F$ is reduced right from the outset. The generic fibre $F_\xi$ is then a finite reduced $R(Z)$-scheme. From $\chara(R(Z))=0$ we conclude that $F_\xi$ is a finite \'etale $R(Z)$-scheme.
By \cite[17.7.11]{EGAIV4} there exists  a non-empty open subset $U$ of $Z$ such that 
$\gamma_U: F_U\to U$ is \'etale. Hence $F_U$ is normal and $\nu_U: F^\nu_U\to F_U$ an isomorphism. Let $B=Z\setminus U$. Then $B$ is a proper closed subset of $Z$ and 
$\gamma\circ \nu(F^\nu(K))\subset \gamma(F(K))\subset \gamma\circ \nu(F^\nu(K))\cup B(K)$. 
The assertion is
immediate from that. 
\end{proof}

\subsection{Preliminaries on near integral points}

{\em Throughout this section let  $S$ be an integral noetherian scheme and $K=R(S)$ its function field.
Let $S_\Zar$ be the set of all non-empty open subsets of $S$. Let $S_\Zar^{(1)}$ be the subset of $S_\Zar$ that consists of all non-empty open subsets $U$ of $S$ such that
$U$ contains all codimension one points of $S$. }

In this section we recall the notion of near integral points and briefly discuss basic properties  that are well-known but hard to find in the literature.

\begin{defi} \label{defi:near} Let $X$ be a $S$-variety and $T$ an integral noetherian $S$-scheme. 
Let $x\in X(R(T))$. We call $x$ near $T$-integral if $x$ lies in the image of the injection $X(\Spec(\OOO_{T,t}))\to X(R(T))$ for every
codimension $1$ point $t$ of $T$. We let $X(T)^{(1)}$ be the set of all near $T$-integral points of $X$.
\end{defi}

\begin{rema} Let $X, Y$ be $S$-varieties and $f: X\to Y$ a morphism. 
Then $f(X(S)^{(1)})\subset f(Y(S)^{(1)})$ and thus the map $f: X(K)\to Y(K)$ induces by restriction
a map $f: X(S)^{(1)}\to Y(S)^{(1)}$. 
\end{rema}

\begin{prop} \label{lemm:vojjav} Let $X$ be an $S$-variety. Let $x\in X(K)$. Then $x\in X(S)^{(1)}$ if and only if there exists $U\in S_\Zar^{(1)}$ such that $x$ lies in the image of the injection $X(U)\to X(K)$. 
\end{prop}

\begin{proof} Assume $x$ near $S$-integral. Then there exists for every codimension $1$ point $s\in S$ an open neighbourhood $U_s$ such that $x$ lies in the image
of $X(U_s)\to X(\OOO_{S,s})\to X(K)$. Let $x_s$ be an element of $X(U_s)$ mapping to $x$. As the maps $X(U_t)\rightarrow X(U_s\cap U_t)\leftarrow X(U_s)$ are injective it
follows that the $x_s$ glue to an element $\hat{x}\in X(U)$ that maps to $x$, where $U=\bigcup U_s$ ($s$ running over all codimension $1$ points of $s$). The other implication is 
trivial. 
\end{proof}

\begin{prop} \label{prop:voj:proper}  Let $S$ be a noetherian normal connected scheme and let $X, Y$ be $S$-varieties. 
Let $f: X\to Y$ be a proper $S$-morphism. Let $x\in X(K)$. If $f(x)\in Y(S)^{(1)}$, then $x\in X(S)^{(1)}$. 
In particular, if $X$ is proper over $S$, then $X(S)^{(1)}=X(K)$. 
\end{prop}

\begin{proof} For every codimension $1$ point $s$ of $S$ the local ring $\OOO_{S,s}$ is a d.v.r., because $S$ is noetherian and normal. By assumption there exists
$y\in Y(\Spec(\OOO_{S,s}))$ such that the composition $$\Spec(K)\to \Spec(\OOO_{S,s})\buildrel y\over \longrightarrow Y$$ is $f(x)$. 
By the valuative criterion of properness there exists a (unique) $\alpha\in X(\Spec(\OOO_{S,s}))$ making the diagram
$$\xymatrix{
\Spec(K)\ar[r]^x\ar[d] & X\ar[d]^f\\
\Spec(\OOO_{S,s})\ar[r]^y\ar@{..>}[ur]^\alpha & Y
}$$
commutative, i.e. $x$ lies in the image of the map $X(\Spec(\OOO_{S,s}))\to X(K)$, as desired. If $X$ is proper over $S$, then we can apply the above with $Y=S$ to conclude that $X(S)^{(1)}=X(K)$. 
\end{proof} 

\begin{defi} \label{defi:HM} A scheme $S$ is said to be a {\bf Hermite-Minkowski scheme} if 
it satisfies the following three conditions.
\begin{enumerate}
\item[(a)] $S$ is  noetherian regular and connected, 
\item[(b)] the function field $R(S)$ has characteristic zero and 
\item[(c)]  for every dense open subscheme $S'$ of $S$
the \'etale fundamental group $\pi_1(S')$ is small, i.e. it has for every $n\in\Nn$ only finitely many open subgroups of index $n$. 
\end{enumerate}
\end{defi}

\begin{rema}\label{rema:hmar}
If $S$ satisfies (a) and (b) and is of finite type over $\Spec(\Zz)$, then $S$ is a Hermite-Minkowski scheme (cf. \cite[Thm. 2.8]{HH}).
\end{rema}

The following proposition yields further examples of Hermite-Minkowski schemes. 

\begin{prop} \label{prop:HM} Let $k$ be a field of characteristic zero and $S$ a smooth
connected $k$-variety. Assume that $\Gal(k)$ is topologically finitely generated (e.g. $k\in\{\Cc,\Rr,\Qq_p\}$). Then $S$ is a Hermite-Minkowski scheme. 
\end{prop}

\begin{proof} It is enough to prove that $\pi_1(U)$ is topologically finitely generated 
for every a dense open subscheme $U$ of $S$. 
Let $k'$ be the algebraic closure of $k$ in $R(U)$. As $U$ is normal, the structure morphism $U\to \Spec(k)$ factors through $\Spec(k')$ and $U$ is geometrically irreducible as $k'$-scheme. $\Gal(k')$ is 
topologically finitely generated because it is an open subgroup of $\Gal(k)$ 
(cf. \cite[Corollary 17.6.3]{FJ}). There is an exact sequence
$1\to \pi_1(U_{\ol{k'}})\to \pi_1(U)\to \Gal(k')\to 1$
(cf. \cite[Prop. 5.6.1]{Szamuely}). 
There exists a finitely generated subfield $k''$ of $k'$ such that $U$ is defined over $k''$ and an embedding $k''\to \Cc$. Then $\pi_1(U_{\ol{k'}})\cong \pi_1(U_{\ol{k''}})\cong \pi_1(U_\Cc)$. It is hence suffices to show that 
$\pi_1(U_\Cc)$ is topologically finitely generated. By \cite[Cor. XII.5.2]{SGA1} $\pi_1(U_\Cc)$ is the 
profinite completion of the topological fundamental group $\pi_1^{\mathrm{top}}(U(\Cc))$, and  $\pi_1^{\mathrm{top}}(U(\Cc))$ is known to be finitely generated by the existence of a good triangulation (cf. \cite{Loj}). Hence $\pi_1(U_\Cc)$ is topologically finitely generated, as desired.
\end{proof}

\begin{prop} \label{prop:hm}
Let $S$ be a Hermite-Minkowski scheme. Let $Y$ be a connected $S$-variety and $f: X\to Y$ a finite \'etale surjective $S$-morphism. 
There exists a finite extension $E/R(S)$ 
such that $Y(S)^{(1)}\subset f(X(E))\cap Y(R(S)).$
\end{prop}

\begin{proof} Let $n$ be the degree of the \'etale cover $X/Y$. 
As 
$\pi_1(S,\os)$ is small 
there exists a connected \'etale cover $T/S$ such that for every \'etale cover $W/S$ of degree $\le n$ the set 
$\Mor_S(T,W)$ is not empty. Let $E:=R(T)$. Let $y\in Y(S)^{(1)}$. Then there exists $S'\in S_\Zar^{(1)}$ such that 
$y\in Y(S')$. Let $C_S$ be the category of all finite \'etale covers of $S$ as a full subcategory of Schemes/$S$. Define $C_{S'}$ similarly. As $S$ is regular and $S'$ contains all codimension $1$ points of $S$ the functor
$$C_S\to C_{S'},\ W\mapsto W\times_S S'$$
is an equivalence of categories by the purity of the branch locus (cf. \cite[Corollaire X.3.3]{SGA1}). Thus $T':=T\times_S S'$ has the following property. For every \'etale cover
$W'/S'$ of degree $\le n$ the set 
$\Mor_{S'}(T',W')$ is not empty. Now $f^{-1}(y)\to S'$ is a (possibly disconnected) \'etale cover of degree $\le n$. Hence 
$\Mor_{S'}(T', f^{-1}(y))\neq \emptyset$. It follows that there exists 
$x\in X(T')$ such that $f(x)=y$. Now $x\in X(T')^{(1)}\subset X(E)$ and $f(x)=y$ as desired.
\end{proof}

\section{Classification of covers}\label{sec:ClassCov}
Throughout this section we work within the following setup.

\begin{setup}\label{setup}
Let $K$ be a field of characteristic zero.
Let $X, Y$ and $Z$  be normal connected $K$-varieties. Let $\xi$ be the generic point of $Z$.  
Let $g: X\to Y$ be a finite surjective $K$-morphism and $f: Y\to Z$ a normal $K$-morphism. Assume that the generic fibre 
of $f$ is geometrically connected. 
By Proposition \ref{prop:stein-np} there exists a non-empty open subset $Z_0$ of $Z$ such that $Y_z=f^{-1}(z)$ is geometrically connected over $k(z)$ and 
$X_z=g^{-1}f^{-1}(z)$ is normal for every $z\in Z_0$. We let $Z_0$ be the largest non-empty open subset of $Z$ with these properties. 
\end{setup}

For $z\in Z_0$ we often consider the
induced morphisms
$$X_z\buildrel g_z\over\longrightarrow Y_z \buildrel f_z\over\longrightarrow \Spec(k(z))$$
where $Y_z$ is a normal geometrically connected $k(z)$-variety and $X_z$ is a normal (possibly disconnected) $k(z)$-variety. 

\begin{defi} 
We call $g$ {\bf vertical over $f$} if $g_\xi: X_\xi\to Y_\xi$ is unramified. 
We call $g$ {\bf split-vertical over $f$} if for some  geometric generic point $\oxi$ of $Z$ the 
finite morphism $g_\oxi: X_\oxi\to Y_\oxi$ is \'etale and has a section. 
We call $g$ {\bf horizontal over $f$} if $g$ is not vertical over $f$. 
We call $g$ {\bf negligible} over $f$ if there exists a normal connected $K$-variety $Z''$, a ramified
finite surjective $K$-morphism $w: Z''\to Z$ and a $Z$-morphism $X\to Z''$. 
\end{defi}

\begin{rema} \label{rema:sv-v}    
 If $g$ is split-vertical over $f$, then $g$ is vertical over $f$. 
\end{rema}

\begin{lemm}   \label{lemm:vh:1}  
If $g$ is vertical over $Z$, then there exists a
dense open subscheme $V$ of $Z$ such that $g_V: X_V\to Y_V$ is \'etale. If $g$ is horizontal over $f$, then
there exists a dense open subscheme $V$ of $Z$ such that $g_z: X_z\to Y_z$ is ramified for all $z\in V$. 
\end{lemm} 

\begin{proof}  
The set $N=\{z\in Z: g_z: X_z\to Y_z\ \mbox{is unramified}\}$ is a  constructible subset $Z$ (cf. \cite[17.7.11]{EGAIV4}). Hence $Z\setminus N$ is also constructible. 
Assume $g$ vertical. 
Then $\xi\in N$. 
Hence there exists a dense open subset $V$ of $Z$ with $V\subset N$, and then $g_V$ is unramified
by \cite[17.8.1]{EGAIV4}. As $Y_V$ is normal, it follows that $g_V$ is \'etale (cf. Remark \ref{rema:basic}). 
Now assume $g$ horizontal. Then
$\xi\in Z\setminus N$. Hence there exists a dense open subset $V$ of $Z$ with $V\subset Z\setminus N$.
It follows that $g_z: X_z\to Y_z$ is ramified for all $z\in V$. 
\end{proof}

We let $u: Z'\to Z$ be the normalization of $Z$ in $X$ and $h: X\to Z'$ the canonical $Z$-morphism. 
Let $Y':=Y\times_Z Z'$ and let $u': Y'\to Y$ and $f': Y'\to Z'$ be the projections of the fibre product. Let $g': X\to Y'$ be the morphism that
makes the diagram

\begin{align}\label{diag:basic}
\xymatrix{
&&&X\ar[dlll]_g\ar@{..>}[dl]_{g'}\ar[ddl]^h & \\
Y\ar[d]_f && Y'\ar[d]_{f'}\ar[ll]^{u'}\\
Z && Z'\ar[ll]_u
}\end{align}
commute. 

\begin{rema} \label{rema:allnormal:1}
By Fact \ref{fact:relnorm} $Z'$ is a normal connected $K$-variety and $u$ a finite surjective morphism. Hence $u$ and $u'$ are finite and surjective. 
The morphism $f$ is normal and the generic fibre of $f$ is geometrically conneced as $R(Z)$-scheme. By base change also $f'$ is normal 
(cf. \cite[6.8.3]{EGAIV2}) and
the generic fibre of $f'$ is geometrically connected as $R(Z')$-scheme. 
By Proposition \ref{prop:stein-np}  
the generic fibre of $h: X\to Z'$ is normal and geometrically connected as $R(Z')$-scheme. 
\end{rema}


\begin{lemm} The $K$-variety $Y'$ is normal and connected and $g'$ is a finite surjective morphism. 
\end{lemm}

\begin{proof} By Remark \ref{rema:allnormal:1} the morphism $f': Y'\to Z'$ is normal, and $Z'$ is a normal connected $K$-variety. 
Thus $Y'$ is normal by \cite[6.8.3]{EGAIV2}. The scheme $Z'$ is irreducible and the morphism $f': Y'\to Z'$ is flat and of finite type, hence open, and its generic fibre is 
(geometrically) connected. Hence $Y'$ is connected. 
Now $u'\circ g'=g$ and the morphisms $u'$ and $g$ are finite. This implies that $g'$ is  a finite
morphism. The generic fibre of $u'$ consists of the generic point of $Y'$ only and $g$ maps the generic point of 
$X$ to the generic point of $Y$. It follows that $g'$ is dominant. But $g'$, being finite, is a closed morphism. Thus $g'$ is surjective.
\end{proof}

\begin{lemm} \label{rema:ufet} There exists a non-empty open subscheme $V$ of $Z$ such that $u_V: Z'_V\to V$ is finite \'etale.  
\end{lemm}

\begin{proof}
The morphism $u: Z'\to Z$ is a finite morphism of normal connected $K$-varieties. Hence $Z'_\xi=\Spec(R(Z'))$. The field extension $R(Z')/R(Z)$ is separable because $\chara(K)=0$. It follows that $u_\xi: Z'_\xi\to \Spec(R(Z))$ is finite \'etale, i.e. $u$ is vertical over $\mathrm{Id}_Z$. By Lemma 
\ref{lemm:vh:1} the assertion follows. 
\end{proof}

\begin{lemm} \label{lemm:horerb} If $g$ is horizontal (resp. vertical) over $f$, then $g'$ is horizontal (resp. vertical) over $f'$.
\end{lemm}

\begin{proof} By Lemma \ref{rema:ufet} the morphism $u_\xi$ is finite \'etale. Hence $u'_\xi$ is finite \'etale (base change). The $R(Z)$-varieties $Y_\xi, Y'_\xi$ and $X_\xi$ are normal and connected (cf. Remark \ref{rema:allnormal:1} and \ref{rema:basic}) and $g_\xi=u'_\xi\circ g'_\xi$. Thus $g_\xi$ is ramified if and only if $g'_\xi$ is ramified. 
The assertion is immediate from that.   
\end{proof}

\begin{lemm} The morphism $g$ is negligible over $f$ if and only if $u$ is ramified. 
\end{lemm}

\begin{proof} Assume $g$ negligible over $f$. Then there exists a normal connected $K$-variety $Z''$, a finite surjective ramified morphism 
$w: Z''\to Z$ and a $Z$-morphism $\alpha: X\to Z''$. By Fact \ref{fact:relnorm} there exists a unique $Z$-morphism $v: Z'\to Z''$ such that $\alpha=v\circ h$. Now $v\circ w=u$ und $u$ and $w$ are finite and surjective. Hence $v$ is finite and surjective. As $w$ is ramified it follows that $u$ must be ramified, too. The reverse implication is trivial. 
\end{proof}

\begin{rema} \label{rema:get} If $g$ is \'etale, then the morphisms $u'$, $u$ and $g'$ are \'etale and $h$ normal. 
\end{rema}

\begin{defi}  
We call $g$ {\bf induced over $f$} if $g'$ is an isomorphism.
\end{defi}

\begin{rema}\label{rema:in-an} If $g$ is induced over $f$ and ramified, then $u'$ and hence also $u$ is ramified and thus $g$ negligible over $f$. 
\end{rema}

\begin{prop}  \label{prop:split-vertical} 
$g$ is split-vertical  over $f$ if and only if  $g$ is induced over $f$. 
\end{prop} 

\begin{proof} 
Let $\oxi: \Spec(\Omega)\to Z$ be a geometric generic point of $Z$. 
It follows from Lemma \ref{rema:ufet} that 
$u_\oxi: Z'_\oxi\to \Spec(\Omega)$
is finite \'etale, and $\Omega$ is an algebraically closed field. Hence $Z'_\oxi$ splits up as 
$Z'_\oxi=\coprod_{j=1}^s \oxi_j'$ where each $\oxi'_j$ is just a copy of $\Spec(\Omega)$ and $u_\oxi|\oxi'_j: \oxi'_j\to \Spec(\Omega)$ is an isomorphism. 
Let $C_j=f_\oxi'^{-1}(\oxi'_j)$ and $D_j=g_\oxi'^{-1}(C_j)$ and consider the diagram

$$\xymatrix{
&&&X_\oxi\ar[dlll]_{g_\oxi}\ar[dl]^{g'_\oxi} \ar@{=}[r]&\coprod_{j=1}^{s} D_j & D_j\ar[l]\ar[dl]\\ 
Y_\oxi\ar[d]_{f_\oxi} && Y'_\oxi \ar[d]_{f'_\oxi}\ar[ll]^{u'_\oxi} \ar@{=}[r]& \coprod_{j=1}^{s} C_j & C_j\ar[l]\ar[d]\\
\Spec(\Omega) && Z'_\oxi\ar[ll]_{u_\oxi}\ar@{=}[r]& \coprod_{j=1}^{s} \oxi'_j & \oxi_j\ar[l]
}$$
in which all three squares are cartesian. It follows that 
$u'_\oxi|C_j: C_j\to Y_\oxi$ is an isomorphism for all $j$ (*). As $Y_\oxi$ is connected, it follows that the $C_j$ are connected.  
The generic fibre of $h$ is geometrically connected, $h=f'\circ g'$ and the $\oxi'_j$ lie over the generic point of $Z'$. 
Hence $D_j=h^{-1}(\oxi'_j)$ is connected. It follows that the $D_j$ (resp. $C_j$) are precisely the connected components of 
$X_\oxi$ (resp. $Y'_\oxi$).

$\Leftarrow$. 
If $g$ is induced over $f$, then $g'$ is an isomorphism, hence $g'_\oxi: X_\oxi\to Y'_\oxi$ is an isomorphism, and from (*) we conclude that 
 $g_\oxi$ must be finite \'etale and have a section, i.e. $g$ is split-vertical over $f$. 

$\Rightarrow$. For the proof of the other implication assume that $g$ is split-vertical over $f$. 
By Remark \ref{rema:sv-v} and Lemma   \ref{lemm:vh:1}  there exists a dense open subscheme $V$ of $Z_0$ such that
$g_V: X_V\to Y_V$ is \'etale. Hence also $g_V': X_V\to Y_V'$ is \'etale (**). 
As the morphism $g_\oxi$ is finite \'etale and has a section 
there exists $j$ such that $g_\oxi|D_j: D_j\to Y_\oxi$ is an isomorphism, and from now on we choose $j$ in that way. 
Via (*) we see that 
$g_\oxi'|D_j: D_j\to C_j$ is an isomorphism (***). All squares in the diagram
$$\xymatrix{
 D_j\ar[r]\ar[d] & X_\oxi\ar[r]\ar[d]^{g'_\oxi} & X_V\ar[r] \ar[d]^{g'_V}& X\ar[d]^{g'}\\
 C_j\ar[r] & Y'_\oxi \ar[r]& Y'_V\ar[r] &Y'
}$$
are cartesian and $g'_V$ is a finite \'etale (see (**)) morphism of normal integral $K$-varieties. Hence $D_j\to C_j$ is finite \'etale of the same degree as $g'_V$. Together with (***) this implies $[R(X):R(Y')]=1$. 
We know that $g': X\to Y'$ is a finite surjective morphism of normal integral $K$-varieties, hence $X$ is the normalization of $Y'$ in $R(X)$. 
Together with $[R(X):R(Y')]=1$ this implies that $g': X\to Y'$ is an isomorphism.
Hence $g$ is induced over $f$, as desired. 
\end{proof}

The above proof of Proposition \ref{prop:split-vertical} is inspired by the proof of  \cite[Thm. X.1.3, p. 261 ff]{SGA1}. 
Note that \cite[Thm. X.1.3, p. 261 ff]{SGA1} could not be applied directly, however, because in the situation of Proposition \ref{prop:split-vertical} the morphism $f$ is neither assumed proper nor it is assumed that {\em all} its fibres are geometrically connected (nor $g$ is assumed \'etale). 

\begin{prop}  \label{prop:horizontal} 
If $g$ is horizontal over $f$, then there exists a non-empty open subscheme $V$ of $Z_0$ such that 
for every $z\in V$ the set $g_z(X_z(k(z)))$ is strongly thin in $Y_z$. 
\end{prop} 

\begin{proof} Lemma \ref{lemm:horerb} implies that $g'$ is horizontal over $f'$. 
By Lemma \ref{lemm:vh:1} there exists a non-empty open subscheme $V'$ of $Z'$ such that $g'^{-1}f'^{-1}(z')\to f'^{-1}(z')$ is ramified for every point $z'\in V'$ (*). After replacing $V'$ by a smaller non-empty open subscheme we can assume that all fibres of $h^{-1}(V')\to V'$ are geometrically connected. 
As $u$ is finite there exists a non-empty open subscheme $V$ of $Z_0$ such that $u^{-1}(V)\subset V'$. By Lemma \ref{rema:ufet} we can, after replacing $V$ by a one of its non-empty open subschemes, assume that $u_V: Z'_V\to V$ is finite \'etale. 

Let $z\in V$. Then $u_z: Z'_z\to \Spec(k(z))$ is finite \'etale. 
Hence $Z'_z$ splits up as 
$Z'_z=\coprod_{j=1}^s z'_j$ where $z'_j=\Spec(E_j)$ with a finite (separable) extension $E_j/k(z)$ for each $j$. 
Let $C_j=f_z'^{-1}(z'_j)$ and $D_j=g_z'^{-1}(C_j)$ and consider the diagram

$$\xymatrix{
&&&X_z\ar[dlll]_{g_z}\ar[dl]^{g'_z} \ar@{=}[r]&\coprod_{j=1}^{s} D_j & D_j\ar[l]\ar[dl]\\ 
Y_z\ar[d]_{f_z} && Y'_z \ar[d]_{f'_z}\ar[ll]^{u'_z} \ar@{=}[r]& \coprod_{j=1}^{s} C_j & C_j\ar[l]\ar[d]\\
\Spec(k(z)) && Z'_z\ar[ll]_{u_z}\ar@{=}[r]& \coprod_{j=1}^{s} z'_j & z'_j\ar[l]
}$$
in which all three squares are cartesian. It follows that $C_j=Y_z\times_{k(z)} \Spec(E_j)$ and that 
$u'_z|C_j: C_j= Y_z\times_{k(z)} \Spec(E_j)\to Y_z$ is just the projection of the fibre product. In particular $u'_z|C_j$ is finite \'etale for every 
$j$ (**). As $Y_z$ is a geometrically connected variety it follows that the $C_j$ are connected. 
The points $z_j'$ are localized in $V'$ by construction. Thus $D_j=h^{-1}(z')$ is connected. 
Hence the $D_j$ (resp. $C_j$) are precisely the connected components of 
$X_z$ (resp. $Y'_z$). By (*) the morphism $D_j\to C_j$ is ramified for every $j$. By (**) it follows that $g_z|D_j: D_j\to Y_z$ is ramified for every $j$. Hence $g_z(X_z(k(z)))=\bigcup_j g_z(D_j((k(z)))$ is strongly thin in $Y_z$. 
\end{proof}

\section{The extension property (EP)}
In Section \ref{sec:splitver} we shall study, in the situation of Setup \ref{setup},  the structure of vertical covers of $f$ under the 
assumption that $f$ satisfies the following technical condition (EP). 

\begin{defi} \label{defi:H} Let $Y, Z$ be connected normal varieties over a field $K$ of characteristic zero. Let $f: Y\to Z$ be a
$K$-morphism. We say that $f$ {\bf satisfies condition (EP)} 
 if for every geometric generic point $\oxi$ of $Z$ and every $\oa\in Y_\oxi(k(\oxi))$ the fibre $Y_\oxi=
Y\times_Z \Spec(k(\oxi))$ is connected and the homomorphism
$i_*: \pi_1(Y_\oxi, \oa)\to \pi_1(Y, i(\oa))$ (induced by the canonical morphism $i: Y_\oxi\to Y$) is injective. 
\end{defi}

The following remark explains that condition (EP) is in fact an extension property for connected \'etale covers of the geometric generic fibre. 

\begin{rema} Let $Y, Z$ be connected normal varieties over a field $K$ of characteristic zero. Let $f: Y\to Z$ be a
$K$-morphism. Assume that the generic fibre of $f$ is geometrically connected. By \cite[Corollary 5.5.8]{Szamuely} the following are equivalent.
\begin{enumerate}
\item[(a)] The morphism $f$ satisfies condition (EP). 
\item[(b)] For every geometric generic point $\oxi$ of $Z$ and every connected \'etale cover $W\to Y_\oxi$ there exists a connected \'etale cover $Y'\to Y$, a connected component $C$ of $Y'\times_Y Y_\oxi$ and a $Y_\oxi$-morphism $C\to W$. 
\end{enumerate}
\end{rema}

Condition (EP) can be checked over a ``big'' open subset of $Z$.

\begin{lemm} \label{lemm:ht-loc} Let  $K$ be a field of characteristic zero. Let $Y$ and $Z$ be connected smooth $K$-varieties and $f: Y\to Z$ be a
a smooth $K$-morphism. Let $D$ be a non-empty open subscheme of $Z$ that contains all codimension $1$ points of $Z$. If $f_D: Y_D\to D$ satisfies condition (EP), then $f: Y\to Z$ satisfies condition (EP). 
\end{lemm} 

\begin{proof} The generic fibre of $f$ agrees with the generic fibre of $f_D$ and thus is geometrically connected. 
Let $\oxi$ be a geometric generic point of $Z$, $\oa\in Y_\oxi(k(\oxi))$ and $\oy:=i(\oa)$. 
Note that $Y_D$ contains all codimension $1$ points of $Y$. Hence $\pi_1(Y_D, \oy)\to \pi_1(Y, \oy)$ is an isomorphism by the 
purity of the branch locus. Furthermore $\pi_1(Y_\oxi, \oa)\to \pi_1(Y_D, \oy)$ is injective because $f_D$ satisfies condition (EP). It follows that
$\pi_1(Y_\oxi, \oa)\to \pi_1(Y, \oy)$ is injective, as desired. 
\end{proof} 

\begin{fact} \label{fact:SGA1ht} Let $Y$ and $Z$ be connected normal varieties over a field $K$ of characteristic zero. Let $f: Y\to Z$  be a smooth proper
morphism. Assume that all fibres of $f$ are geometrically connected and that $f$ has a section. Let $\oxi$ be a geometric generic point of $Z$ and $a\in Y_\oxi(k(\oxi))$. Then there is an exact sequence
$$0\to \pi_1(Y_\oxi, a)\to \pi_1(Y, i(a))\to \pi_1(Z, \oxi)\to 0$$
by \cite[XIII.4.3, XIII.4.4]{SGA1}. In particular $f$ satisfies condition (EP). 
\end{fact}

\begin{coro} \label{coro:SGA1ht} Let $Y$ and $Z$ be connected smooth varieties over a field $K$ of characteristic zero. Let $f: Y\to Z$  be a smooth proper
morphism. If $Y(R(Z))\neq \emptyset$, then all fibres of $f$ are geometrically connected and $f$ satisfies condition (EP). 
\end{coro}

\begin{proof} Let $\xi$ be the generic point of $Z$. By Corollary \ref{coro:conn} all fibres of $f$ are geometrically connected. 
Let $\sigma\in R(Z)$. Then $\sigma$ extends to a section $s: D\to Y$ of $f$ defined on an open subset $D$ of $Z$ that contains all codimension $1$ points of $Z$. By Fact \ref{fact:SGA1ht} $f_D: Y_D\to D$ satisfies condition (EP). By Lemma
\ref{lemm:ht-loc} $f$ satisfies condition (EP). 
\end{proof}

\begin{lemm} \label{lemm:EPprod} Let $K$ be a field of characteristic zero. Let $Z_1, Z_2$ be geometrically connected normal $K$-varieties and $f: Z_1\times Z_2\to Z_2$ the projection. Then $f$ satisfies property (EP). 
\end{lemm}

\begin{proof} Let $\Omega$ be an algebraically closed field and $z: \Spec(\Omega)\to Z_2$ be a geometric point of $Z_2$. 
Let $Y=Z_1\times Z_2$. 
Let $y$ be a geometric point of $Y_z$. The canonical morphism $Y_z\buildrel j\over\longrightarrow Y$ factors as $i: Y_z\to Y_\Omega\buildrel p \over \longrightarrow Y$. I claim that $i_*: \pi_1(Y_z, y)\to \pi_1(Y, i(y))$ is injective. As $p_*: \pi_1(Y_\Omega, j(y))\to \pi_1(Y, i(y))$ is injective, it is enough to show that $j_*: \pi_1(Y_z, y)\to \pi_1(Y_\Omega, j(y))$ is injective. But $j: Y_z=Z_{1,\Omega}\to Y_\Omega=Z_{1,\Omega}\times Z_{2,\Omega}$ is a section of the projection $g: Z_{1,\Omega}\times Z_{2,\Omega}\to Z_{1,\Omega}$, i.e. $g\circ j=\mathrm{Id}$. It follows that 
$g_*\circ j_*=Id$. Thus $j_*$ is injective, as desired.
\end{proof}

We shall see in Section \ref{sec:digression} further and more general situations in which condition (EP) is satisfied. 


\section{Splitting vertical covers}\label{sec:splitver}

For normal geometrically connected varieties $U$ and $V$ over a field of characteristic zero and $p: U\times V\to V$ is the projection, the structure of covers of $U\times V$ that are vertical over $p$ was understood in \cite[Lemma 2.1]{Lug3}. See \cite[Lemma 2.17]{CDJLZ} for the proper case. 
In this section we shall establish, in a situation where $Y, Z$ and $f: Y\to Z$ are as in Setup \ref{setup}, similar information about covers of $Y$ that are vertical over $f$ provided $f$ satisfies condition (EP). 

\begin{prop} \label{prop:split-neg:1} Let $K,X,Y,Z,\xi, f, g$ and $Z_0$ be as in Setup \ref{setup}.
Assume that $f$ is satisfies condition (EP). Assume that $g$ is vertical over $f$ and ramified.
There exist $K$-varieties $X', X'', Y'$ and $Z'$ and a commutative diagram
$$\xymatrix{
X\ar[d]_g & X' \ar[d]_{g'}\ar[l]_h& X''\ar[dl]^{g''}\ar[l]_{\mathrm{incl}}\\
Y\ar[d]_f & Y' \ar[d]_{f'}\ar[l]_v&\\
Z & Z'\ar[l]_u
}$$
with the following properties.
\begin{enumerate}
\item[(a)] $u: Z'\to Z$ is the relative normalization of $Z$ in $Y'$ and $f': Y'\to Z'$ the canonical $Z$-morphism. 
\item[(b)] $X'=X\times_Y Y'$ and $h$ and $g'$ are the projections of that fibre product. 
\item[(c)] The morphisms $u, v$ and $h$ are finite \'etale and surjective.
\item[(d)] The $K$-varieties $Z'$ and $Y'$ are normal and connected; the $K$-variety $X'$ is normal (and possibly disconnected).
\item[(e)] The morphism $f'$ is normal and its generic fibre is geometrically connected. Moreover $g'$ is finite and surjective. 
\item[(f)] The $K$-variety $X''$ is a connected component of $X'$ and
$\mathrm{incl}$ the inclusion. 
\item[(g)] The morphism $g''$ is finite, surjective and split-vertical (and thus induced) over $f'$. 
\item[(h)]  If $g$ is ramified, then $g''$ is ramified. 
\end{enumerate}
\end{prop}

\begin{proof} 
Let $\oxi$ be a geometric generic point of $Z$ and $a\in Y_\oxi(k(\oxi))$. The morphism $g_\oxi: X_\oxi\to Y_\oxi$ is finite \'etale because $g$ is vertical over $f$. As $f$ satisfies condition (EP) there exist a connected \'etale cover $v: Y'\to Y$, 
a connected component $C$ of $Y'_\oxi=Y_\oxi\times_Y Y'$  and 
a $Y_\oxi$-morphism $\sigma: C\to X_\oxi$. Note that $Y'$ is a connected normal $K$-variety (*). We define $Z', u, f', X', h$ and $g'$ so that (a) and (b) become true. The morphism $v: Y'\to Y$ is finite \'etale surjective by construction. Hence also $h$ is finite \'etale surjective by base change. The morphism $u$ is finite \'etale surjective (cf. Remark \ref{rema:get}). Hence (c) is true. Now $Z'$ is connected (cf. Fact \ref{fact:relnorm}) and  \'etale over the normal $K$-variety $Z$, hence $Z'$ is normal. The $K$-variety $X'$ is \'etale over the normal $K$-variety $X$. Thus $X'$ is normal. This proves (d). The generic fibre of $f'$ is geometrically connected by construction and $f'$ is normal by Remark \ref{rema:get}. 
The morphism $g$ is finite and surjective. Thus $g'$ is also finite and surjective by base change. Hence we have (e). 

The morphism $g_\oxi: X_\oxi\to Y_\oxi$ is finite \'etale surjective ($g$ is vertical) and thus also 
$g'_\oxi: X'_\oxi\to Y'_\oxi$ is finite \'etale surjective by base chance. 
The $Y_\oxi$-morphism $\sigma: C\to X_\oxi$ induces a morphism $s: C\to g'^{-1}_\oxi(C)$ that makes the diagram
$$\xymatrix{
 & & & C\ar[dlll]_\sigma\ar[ddl]^{\mathrm{Id}_C}\ar@{..>}[dl]_s\\
 X_\oxi\ar[d]_{g_\oxi} & X'_\oxi\ar[l]^{h_\oxi}\ar[d]^{g'_\oxi} & g'^{-1}_\oxi(C)\ar[l]\ar[d]\\
 Y_\oxi & Y'_\oxi\ar[l]_{v_\oxi} & C\ar[l]
}$$
commute, i.e. $s$ is a section of the finite \'etale morphism $g'^{-1}_\oxi(C)\to C$. Thus there exists a connected component 
$D$ of $g'^{-1}_\oxi(C)$ such that $g'_\oxi|D: D\to C$ is an isomorphism (*). There exists a connected component $X''$ of $X'$ such that $D$ is a connected component of $X''_\oxi$ and we let $g''=g'|X''$. Then (f) is true by construction. 
Moreover $g''_\oxi$ is finite \'etale surjective (**). We next prove that (g) holds true for that choice of $X''$. 

The morphism $\mathrm{incl}$ is finite. It follows that $g''=g'\circ \mathrm{incl}$ is finite. The morphism $h\circ \mathrm{incl}$ is finite \'etale, hence open and closed. As $X$ is connected it follows that $h\circ \mathrm{incl}$ is surjective. As $g$ is surjective we see that $X''$ is dominant over the integral $K$-variety $Y$. Only the generic point of $Y'$ is mapped to the generic point of $Y$ under $v$ by (c). Thus $g''$ must be dominant. As $g''$ is also a closed morphism, it follows that $g''$ is surjective. We next show that $g''$ is split-vertical over $f'$. For this we consider the commutative diagram
$$\xymatrix{
 &    & X_\oxi\ar[ld]\ar[ddd] && X'_\oxi\ar[ld]\ar[ll]\ar[ddd]&& X''_\oxi \ar[ddd]\ar[ll]\ar[llld] && D\ar[ll]\ar@{>->>}[llld]\\
 &Y_\oxi\ar[ld] \ar[ddd] & &Y'_\oxi\ar[ll] \ar[ddd]\ar[ld]&&C\ar[ll]\\
\oxi \ar[ddd] && Z'_\oxi\ar[ll]\ar[ddd]\\
 &    & X\ar[ld]_g & &X' \ar[ld]_{g'}\ar[ll]_h&&X''  \ar[ll]_{\mathrm{incl}}\ar[llld]^{g''}\\
 &Y \ar[ld]_f & &Y'\ar[ll]_v\ar[ld]_{f'} & &\\
Z & & Z'\ar[ll]_u && \\  
}.$$

The scheme $Z'_\oxi$ is finite \'etale over $\oxi=\Spec(\Omega)$. Thus $Z'_\oxi=\prod_{j=1}^s \oxi'_j$ where each $\oxi'_j$ is just a copy of $\Spec(\Omega)$. Hence $Y'_\oxi=\prod_{j=1}^s 
f_\oxi'^{-1}(\oxi'_j)$ and the $f_\oxi'^{-1}(\oxi'_j)$ are connected by (e). Hence there exists $j$ such that $C=f_\oxi'^{-1}(\oxi'_j)$. 
We can view $\oxi_j'$ as a geometric point of $Z'$. By (*) and (**) it follows that $g''^{-1}f'^{-1}(\oxi_j')\to f'^{-1}(\oxi_j')=C$ is finite \'etale and has a section, i.e. that $g''$ is split-vertical over $f'$. Proposition \ref{prop:split-vertical} implies that $g''$ is induced over $f'$. 
Hence we have shown (g).

For the proof of (h) assume that $g$ is ramified. Recall that $X, X''$ and $Y$ are normal connected varieties and that $h'=h\circ incl$, $g''$ and $v$ are finite and surjective. As $g$ is ramified, $g\circ h'$ must be ramified, and as $g\circ h'=v\circ g''$ it follows that $g''$ is ramified. 
\end{proof}

\begin{prop} \label{prop:split-neg} Let $K,X,Y,Z,\xi, f, g$ and $Z_0$ be as in Setup \ref{setup}.
Assume that $f$ is satisfies condition (EP). If $g$ is vertical over $f$ and ramified, then 
there exist connected normal $K$-varieties $X''$ and $Z''$, a finite \'etale morphism $X''\to X$ 
a finite surjective ramified morphism $Z''\to Z$ and a dominant $Z$-morphism $X''\to Z''$. 
\end{prop}

\begin{proof} Consider the data from Proposition \ref{prop:split-neg:1}. Let $w: Z''\to Z'$ be the normalization of $Z'$ in $X''$ and $\ell: X''\to Z''$ the
canonical morphism. $Z''$ is a normal connected $K$-variety, $w$ is finite and surjective and $\ell$ is dominant. 
By the universal mapping property of the fibre product the commutative diagram from Proposition \ref{prop:split-neg:1} can be extended to a  commutative diagram
 $$\xymatrix{
X\ar[d]_g & X' \ar[d]_{g'}\ar[l]_h& X''\ar[dl]_{g''}\ar[l]_{\mathrm{incl}}\ar[d]_\alpha\\
Y\ar[d]_f & Y' \ar[d]_{f'}\ar[l]_v& Y'\times_{Z'} Z''\ar[d]_q\ar[l]_p\\
Z & Z'\ar[l]_u & Z''\ar[l]_w 
}$$
such that $q\circ \alpha=\ell$. By Proposition \ref{prop:split-neg:1}(e) the morphism $\alpha$ is an isomorphism and by 
Proposition \ref{prop:split-neg:1}(f) the morphism $w$ is ramified. It follows that $u\circ w$ is finite surjective and ramified. Moreover $h\circ \mathrm{incl}$ is finite \'etale surjective. The assertion thus follows.
\end{proof}

\begin{defi} In the situation of Setup \ref{setup} we say that $g$ is {\bf almost negligible} over $f$ if 
 there exists a normal connected $K$-variety $X'$ and a finite \'etale $K$-morphism $h: X'\to X$ such that $g\circ h$ is negligible over $f$. 
\end{defi}

\begin{coro} \label{coro:split-neg}  Let $K,X,Y,Z,\xi, f, g$ and $Z_0$ be as in Setup \ref{setup}.
Assume that $f$ satisfies condition (EP). If $g$ is vertical over $f$ and ramified, then $g$ is almost negligible over $f$.
\end{coro}

\begin{proof} This is immediate from Proposition \ref{prop:split-neg}.
\end{proof}

\section{Pullback of covers along sections}

\begin{lemm} \label{lemm:negligible-sec}Let $K,X,Y,Z,\xi, f, g$ and $Z_0$ be as in Setup \ref{setup}. Let $D$ be an open subscheme of $Z$ 
that contains all codimension $1$ points of $Z$. Let
$s\in \Mor_Z(D,Y)$ and $T:=\{z\in D(K): s(z)\in g(X(K))\}$. 
If $Z$ is smooth and $g$ is almost negligible over $f$, then  $T$ is 
strongly thin in $D$. 
\end{lemm}

\begin{proof} Assume that $g$ is almost negligible. Then there exist normal connected $K$-varieties $X'$ and $Z'$, 
a finite \'etale $K$-morphism $h: X'\to X$, a ramified finite surjective $K$-morphism $u: Z'\to Z$ 
and a dominant $Z$-morphism $f': X'\to Z'$. As $Z$ is regular and $D$ contains all codimension $1$ points of $Z$, purity of the branch locus
implies that $u_D: Z'_D\to D$ is ramified. 
We put $F=g^{-1}(s)$ and $F'=h^{-1}g^{-1}(s)$. We identify $s$ with an element of $\Mor_D(D, Y_D)$  and consider the 
diagram
$$\xymatrix{
F'\ar[r]^{h'}\ar[d]^{s'} & F\ar[r]^{g'}\ar[d]& D\ar[d]^s\\
X'_D\ar[r]^{h_D}\ar[d]^{f'_D} & X_D\ar[r]^{g_D} & Y_D\ar[d]^{f_D}\\
Z'_D\ar[rr]^{u_D} && D
}
$$
where the two upper squares are cartesian and $f_D\circ s=\mathrm{Id}_D$. 
The morphism $g'$ is finite and surjective by base change
because $g_D$ is finite and surjective. The morphism $h'$ is finite \'etale surjective because $h_D$ has these properties. 
Let $\nu: F^\nu\to F$ be the normalization of $F$ and $g''=g'\circ \nu$. Note that $g''$ is finite and surjective. 
Note that $T=g'(F(K))$ and we have to prove that this set is strongly thin in $D$. 
By Lemma \ref{lemm:disconcover} it is suffices to prove that $g''(F^\nu(K))$ is strongly thin in $D$. For this it is enough to show that $g''(C(K))$ is strongly thin in $D$ for every connected component $C$ of $F^\nu$. So let $C$ be such a connected component. If $g''|C: C\to D$ is not surjective then $g''(C(K))$ is contained in a proper closed subset of $D$, hence strongly thin in $D$. We can therefore assume that $g''|C: C\to D$ is surjective. It suffices to prove that  $g''|C: C\to D$ is ramified. 
The projection $h'': C\times_F F'\to C$ is finite \'etale surjective because $h'$ has these properties. We let $C'$ be a connected component of the normal $K$-variety $C\times_F F'$ and note that $h''|C': C'\to C$ is finite \'etale surjective. Thus it is enough to prove that the finite surjective morphism $w'':=(g''|C)\circ (h''|C')$ is ramified. We have a commutative diagram

$$\xymatrix{
C'\ar[rr]^{w''} \ar[d] & & D\\
Z'_D \ar[rru]_{u_D}
}
$$

of normal connected varieties. The vertical morphism is finite and surjective because $u_D$ and $w''$ are finite and surjective. As $u_D$ is ramified it follows that $w''$ is ramified, as desired. 
\end{proof}

\begin{prop} \label{prop:sec:good:pts} Let $K,X,Y,Z,\xi, f, g$ and $Z_0$ be as in Setup \ref{setup}. 
Let $D$ be an open subscheme of $Z$ 
that contains all codimension $1$ points of $Z$. 
Assume that $f$ satisfies condition (EP) and that $Z/K$ is smooth. 
Let
$s\in \Mor_Z(D,Y)$ and $T:=\{z\in D(K): s(z)\in g(X(K))\}$. 

\begin{enumerate}
\item[(a)] If $g$ is vertical over $f$ and ramified, then $T$ is 
strongly thin in $D$. 
\item[(b)] If the point $s_\xi\in Y_\xi(R(Z))$ does not lie in $g_\xi(X_\xi(R(Z)))$, then 
$T$ is thin in $D$. 
\end{enumerate}
\end{prop}

\begin{proof}[Proof of Part (a)] 
By Proposition \ref{prop:split-neg} $g$ is almost negligible over $f$. Hence we can apply Lemma \ref{lemm:negligible-sec} to conclude that
$T\cap D(K)$ is strongly thin in $D$. 
\end{proof} 

\begin{proof}[Proof of Part (b)] Assume that  $s_\xi\notin g_\xi(X_\xi(R(Z)))$ (*).
Let $F:=g^{-1}(s)=X\times_{Y,s} D$ and $\gamma: F\to D$ the projection of the fibre product. Note that $T=\gamma(F(K))$. 
It is enough to prove that $T$ is thin in $D$. The diagrams
$$\xymatrix{F\ar[d]\ar[r]^\gamma & D\ar[d]^s & F_\xi\ar[d]\ar[r] & \xi\ar[d]\\
X\ar[r]^g&Y\ar[d]^f&X_\xi\ar[r]& Y_\xi\ar[d]\\
& Z && \xi
}$$
($\xi=\Spec(R(Z))$) are commutative and both squares cartesian. (*) implies that $F_\xi(R(Z))=\emptyset$ (**). 

Let $\nu: F^\nu\to F$ be the normalization of $F$. 
By Lemma \ref{lemm:disconcover} it suffices to show that $\gamma\circ \nu (F^\nu(K))$ is thin. 
The morphism $\gamma\circ \nu$ is finite. For this we have to prove that those connected components 
$[R(I):R(Z)]\ge 2$ for those connected components $I$ of $F^\nu$ that are dominant over $D$. Let $I$ be such a connected component. Then $I_\xi$ is normal and connected, hence integral, and finite surjective over $D$. 
We have $I_\xi=\Spec(R(I))$ and $I_\xi$ admits a $R(Z)$-morphism $I_\xi\to F_\xi$. By (**) we have $F_\xi(R(Z))=\emptyset$.
It follows that $I_\xi(R(Z))=\emptyset$. Thus there does not exist a $R(Z)$-algebra homomorphism $R(Z)\to R(I)$. It follows that 
$[R(I):R(Z)]\ge 2$, as desired.
\end{proof}

\section{Proof of Theorem \ref{thm:1}}\label{sec:main}

We shall establish Theormem \ref{thm:1} by combining the following Theorem \ref{thm:main:abstract} with information coming from \'etale homotopy theory (cf. Fact \ref{fact:SGA1ht}, Corollary \ref{coro:SGA1ht}).

\begin{thm}\label{thm:main:abstract}
Let $K$ be a field of characteristic zero. Let $Y$ and $Z$ be connected smooth $K$-varieties. 
Let $f: Y\to Z$ be a smooth morphism. Assume that $f$ satisfies condition (EP).  
Assume that $Z$ has HP over $K$ and that the subset $Y(Z)^{(1)}$ of $Y(R(Z))$ is not strongly thin in $Y_{R(Z)}$. Then 
$Y$ has WHP over $K$. 
\end{thm}

\begin{proof} Let $(X_i)_{i\in I}$ a finite family of normal connected $K$-varieties and $(g_i: X_i\to Y)_{i\in I}$ a finite family of 
finite surjective ramified morphisms. Let $V\subset Y$ be a non-empty open subset. We have to prove
that there exists a point $v\in V(K)$ such that $v\notin \bigcup_{i\in I} g_i(X_i(K))$. The generic fibre $Y_\xi$ of $f$ 
has WHP, hence it is geometrically irreducible over $R(Z)$. Let $I^h$ (resp. $I^v$) be the set of all $i\in I$ such that 
$g_i$ is horizontal (resp. vertical) over $f$.  For every $i\in I_h$ the morphism $g_{i, \xi}: X_{i, \xi}\to Y_\xi$ is ramified.
Here $Y_\xi$ is $R(Z)$-variety that has WHP and $X_{i, \xi}$ a normal connected $R(Z)$-variety. Hence the subset 
$M:=\bigcup_{i\in I_h} g_{i, \xi}(R(Z))$
of $Y_\xi(R(Z))$ is strongly thin in $Y_\xi$. As the subset  $Y(Z)^{(1)}$ of $Y(R(Z))=Y_\xi(R(Z))$ is not strongly in $Y_\xi$, we can choose 
a point $\sigma\in Y(Z)^{(1)}\cap V_\xi(R(Z))$ such that $\sigma \notin M$. That point $\sigma$ extends to a $Z$-morphism $s: D\to Y$ for some non-empty open subscheme $D$ of $Z$ that contains all codimension $1$ points of $Z$ (cf. Proposition \ref{lemm:vojjav}). For $i\in I$ we define
$$T_i:=\{z\in D(K): s(z)\in g(X(K))\}.$$
Proposition \ref{prop:sec:good:pts}(b) implies that $T_i$ is thin in $D$ for every $i\in I^h$. On the  other hand, 
by Proposition \ref{prop:sec:good:pts}(a), $T_i$ is (strongly) thin in $D$ for every $i\in I^v$. Hence $T_i$ is a thin subset of $D$ for every $i\in I$. 
The set $T:=\bigcup_{i\in I} T_i$ is thin in $D$. As $\sigma\in V_\xi(R(Z))$ it follows that $s^{-1}(V)$ is a non-empty open subset of $D$. As $Z$ (and thus also $D$) satisfies HP 
we can choose $z\in s^{-1}(V)(K)\setminus T$. Then $s(z)\in V(K)$ and $s(z)\notin  \bigcup_{i\in I} g_i(X_i(K))$, as desired. 
\end{proof}

\begin{coro} \label{coro:main:smoothproper}
Let $K$ be a field of characteristic zero. Let $Y$ and $Z$ be connected smooth $K$-varieties. 
Let $f: Y\to Z$ be a smooth proper morphism. 
Assume that $Z$ has HP over $K$ and that the generic fibre $Y_{R(Z)}$ has WHP over $R(Z)$. Then 
$Y$ has WHP over $K$. 
\end{coro}

\begin{proof} We have $Y(Z)^{(1)}=Y(R(Z))$ because $f$ is proper. 
As $Y_{R(Z)}$ has WHP over $R(Z)$ it follows that $Y(Z)^{(1)}$ is not strongly thin in $Y_{R(Z)}$. In particular $Y(R(Z))\neq \emptyset$. 
Thus $f$ satisfies condition (EP) (cf. Corollary \ref{coro:SGA1ht}, Fact \ref{fact:SGA1ht}) and all fibres of $f$ are geometrically connected. 
Now Theorem \ref{thm:main:abstract} implies the assertion. 
\end{proof}

\begin{coro} \label{coro:as2} 
Let $K$ be a field of characteristic zero. Let $Z$ be a connected smooth $K$-variety, $F=R(Z)$ the function field of $Z$  and $A/Z$ an abelian scheme. 
Assume that $A(F)$ is Zariski dense in $A_{F}$
and that $Z$ has HP over $K$. If the Chow trace $\mathrm{Tr}_{\oF/\oK}(A_\oF)$ is zero, then $A$ has WHP over $K$. 
\end{coro}

\begin{proof} The morphism $f: A\to Z$ is smooth and proper. The generic fibre $A_{F}$ has WHP over $F$ by \cite[Theorem B]{Jav2}.
The assertion hence follows by Corollary \ref{coro:main:smoothproper}. 
\end{proof}

\begin{coro} \label{coro:as} Let $K$ be a finitely generated field of characteristic zero. Let $Z$ be a connected smooth $K$-variety,
$F=R(Z)$ the function field of $Z$ and $A/Z$ an abelian scheme. 
Assume that $A(F)$ is Zariski dense in $A_{F}$ and  that $Z$ has HP over $K$. Then $A$ has WHP over $K$. 
\end{coro}

\begin{proof} The morphism $f: A\to Z$ is smooth and proper. The generic fibre $A_{R(Z)}$ has WHP over $R(Z)$ by the main result of \cite{CDJLZ}.
The assertion hence follows by Corollary \ref{coro:main:smoothproper}. 
\end{proof}

\begin{defi} Let $K$ be a field and $Z/K$ a variety. Let $A/Z$ be an abelian scheme. 
We say that $A$ is {\bf $Z/K$-constant} if there exists an abelian variety $B/K$ such that $A\cong B\times_K Z$. 
\end{defi}

The content of the following Remark was pointed out by Arian Javanpeykar. 

\begin{rema} Let $K$ be a finitely generated field of characteristic zero. Let $Z$ be a connected smooth $K$-variety
and $A/Z$ an abelian scheme. 
Assume that $A(K)$ is Zariski dense in $A$ and that $Z$ has HP over $K$. 
\begin{enumerate}
\item[(a)] If $A$ is $Z/K$-constant, then the mixed fibration theorem of Luger (cf. \cite{Lug4}) implies that $A$ has WHP over $K$. 
\item[(b)] If $\pi_1(Z_\oK)=1$, then $A$ is $Z/K$-constant by the following Proposition. 
\item[(c)] If $Z/K$ is proper, then $\pi_1(Z_\oK)=1$ by a theorem of Corvaja and Zannier (cf. \cite{CDJLZ}). 
\end{enumerate}
\end{rema}

Thus the main merit of Corollary \ref{coro:as} lies in its ability to prove WHP for certain non-constant abelian schemes.

\begin{prop} \label{prop:ascheme} Let $K$ be a field of characteristic zero. Let $Z$ be a connected smooth $K$-variety and $A/Z$ an abelian scheme. 
If $\pi_1(Z_\oK)=1$ and $Z(K)\neq \emptyset$, then $A$ is $Z/K$-constant. 
\end{prop}

\begin{proof} Let $\ell$ be a rational prime. We choose $z\in Z(K)$ and and a geometric point $\oz: \Spec(\Omega)\to Z$ localized at $z$. 
The Tate module $T_\ell(A):=\plim_{n\in\Nn} A[\ell^n]_\oz(\Omega)$ is a $\pi_1(Z, \oz)$-module in a natural way. The point $z: \Spec(K)\to Z$ 
induces a section $z_*: \pi_1(z, \oz)\to \pi_1(Z, \oz)$ of the right hand map in the short exact sequence
$$1\to \pi_1(Z_\oK, \oz)\to \pi_1(Z, \oz)\to \pi_1(z, \oz)\to 1$$
and $T_\ell(A)$ thus becomes a $\pi_1(z, \oz)$-module via $z_*$. Consider the abelian scheme  $B=A_z\times_K Z$ over $Z$. 
The isomorphism $u_z=\mathrm{Id}_{A_z}: A_z\to B_z$ induces a $\pi_1(z, \oz)$-equivariant isomorphism $u_\ell: T_\ell(A)\to T_\ell(B)$. As $\pi_1(Z_\oK, \oz)=1$ by our assumption it follows that $z_*$ is an isomorphism and thus $u_\ell$ is even $\pi_1(Z, \oz)$-equivariant. 
Now \cite{groth} implies that the isomorphism $u_z: A_z\to B_z$ is induced by an isomorphism $u: A\to B$ of abelian schemes.
\end{proof}

\section{Proof of Theorem \ref{thm:2}}
In this section we shall frequently work within the following setup.

\begin{setup}\label{setup2}
Let $S$ be a Hermite-Minkowski scheme and $K=R(S)$ its function field. Let $\YY$ and $\ZZ$ be $S$-varieties, $F: \YY\to \ZZ$ an $S$-morphism, $Y=\YY_K$, $Z=\ZZ_K$ and $f=F_K$. Assume that $Y$ and $Z$ are normal connected $K$-varieties, that the morphism $f: Y\to Z$ is normal and that all fibres of $f$ are geometrically connected. Let $\Sigma$ be a subset of $\YY(S)^{(1)}$ and let $\Xi$ be the set of all $z\in \ZZ(S)^{(1)}$ such that $\Sigma\cap Y_z(K)$ is not strongly thin. Let $\Xi_0$ be the set of all $z\in\ZZ(S)^{(1)}$ such that $\YY_z$ has WHP over $S$. 
\end{setup}

For every point $z\in \ZZ(S)^{(1)}$ we denote by $D(z)$ the largest open subset of $S$ such that $z$ extends to an element $\hat{z}$ of $\ZZ(D(z))$. We let $\YY_z=\YY\times_{\ZZ, F, \hat{z}} D(z)$ and $Y_z=Y\times_{Z,f,z} \Spec(K)$ and note that $Y_z$ is the generic fibre of $\YY_z\to D(z)$. 
Note that $D(z)$ contains all codimension $1$ points of $S$ and thus $\YY_z(S)^{(1)}=\YY_z(D(z))^{(1)}=Y_z(K)\cap \YY(S)^{(1)}$. 
Of course we often simply write $z$ instead of $\hat{z}$. 

\begin{lemm} \label{lemm:weil} 
In the situation of Setup \ref{setup2} let $X$ be a normal connected $K$-variety and $g: X\to Y$ a finite surjective $K$-morphism. Assume that $g$ is vertical and ramified. Consider the subset 
$\Gamma=g(X(K))\cap \YY(S)^{(1)}$ of $Y(K)$. If $f$ satisfies condition (EP), then $f(\Gamma)$ is strongly thin in $Z$. 
\end{lemm}

\begin{proof} By Corollary \ref{coro:split-neg} there exist normal connected $K$-varieties $X'$ and $Z'$, a finite \'etale morphism $h: X'\to X$, a finite
surjective ramified morphism $u: Z'\to Z$ and a dominant $Z$-morphism $w: X'\to Z'$. 
After replacing $S$ by one of its open subschemes the $K$-varieties $X$ and $X'$ extend to $S$-varieties $\XX$ and $\XX'$, and the $K$-morphisms $g$ and $h$ extend to $S$-morphisms $G: \XX\to \YY$ and $H: \XX'\to \XX$. After replacing $S$ once more we can assume that $G$ is finite and surjective and $H$ is finite \'etale surjective. Proposition \ref{prop:hm} implies that there exists a finite extension $E/K$ such that 
$\XX(S)^{(1)}\subset h(X'(E))\cap X(K)$. 
The set $u(Z'(E))\cap Z(K)$ is strongly thin in $Z$ by \cite[Lemma 2.1]{BFP3}.
It is hence enough to prove that $f(\Gamma)\subset u(Z'(E))\cap Z(K)$. Let $z\in f(\Gamma)$. It is clear that $z\in Z(K)$. We prove that 
$z\in u(Z'(E))$: There exists $x\in X(K)$ such that $g(x)\in \YY(S)^{(1)}$ and such that $f(g(x))=z$. 
By Proposition \ref{prop:voj:proper} we have 
$x\in \XX(S)^{(1)}$. 
By our choice of $E$ there is $x'\in X'(E)$ such that $h(x')=x$. It follows that 
$z=f(g(h(x')))=u(w(x'))\in u(Z'(E))$, as desired. 
\end{proof}

\begin{thm} \label{thm:iwhp} Consider the situation of Setup \ref{setup2}. Assume that $f$ satisfies condition (EP).   
If $\Xi$ is not strongly thin in $Z$, then
$\Sigma$ is not strongly thin in $Y$. 
\end{thm}

\begin{rema} \label{rema:iwhp} Consider the situation of Theorem \ref{thm:iwhp}
\begin{enumerate}
\item[(a)] If $f$ is proper, smooth and has a section, then $f$ satisfies condition (EP) by Fact \ref{fact:SGA1ht}, and thus the theorem applies.
\item[(b)] If $f$ is proper, then $\Xi_0=\{z\in \ZZ(S)^{(1)}|\mbox{$Y_z$ has WHP over $K$}\}$. 
\end{enumerate} 
\end{rema}

\begin{proof}[Proof of Theorem \ref{thm:iwhp}]
Let $(X_i)_{i\in I}$ be a finite family of normal connected $K$-varieties and $(g_i: X_i\to Y)_{i\in I}$ a finite family of 
finite surjective ramified morphisms and $C$ a proper closed subset of $Y$. 
It is enough to show that $\Sigma$ is not contained in $C(K)\cup\bigcup_{i\in I} g_i(X_i(K))$. 
Let $V=Y\setminus C$. Let $\Gamma_i=g_i(X_i(K))\cap \YY(S)^{(1)}$. 
As $\Sigma\subset \YY(S)^{(1)}$ it is enough to prove that $(\Sigma\cap V(K))\setminus \left( \bigcup_{i\in I} \Gamma_i\right)$ is not empty. 

Let $I^h$ (resp. $I^v$) be the set of all $i\in I$ such that 
$g_i$ is horizontal (resp. vertical) over $f$. We put $\Gamma_i(z):=\Gamma_i \cap Y_z(K)$,  $\Gamma^h(z)=\bigcup_{i\in I^h} \Gamma_i(z)$,
$\Gamma^v(z)=\bigcup_{i\in I^v} \Gamma_i(z)$ and $\Gamma(z)=\bigcup_{i\in I} \Gamma_i(z)$. 

By Proposition \ref{prop:horizontal} there exists a non-empty open subset $U$ of $Z$ such that
for every $z\in U(K)$ and all $i\in I^h$ the $K$-variety $Y_z$ is connected and the set
$\Gamma_i(z)$ is strongly thin in $Y_z$. Thus $\Gamma^h(z)$ is strongly thin in $Y_z$ for every $z\in U(K)$ (*). 
As $f$ is smooth $f(V)$ is open in $Z$. Let $W=f(V)\cap U$. 
By Lemma \ref{lemm:weil} the set $\Delta:=\bigcup_{i\in I^v} f(\Gamma_i)$ is strongly thin 
in $Z$. We can pick $z\in (\Xi\cap W(K))\setminus \Delta$ because $\Xi$ is not strongly thin in $Z$. 

From $z\in W(K)\subset U(K)$ and from (*) we conclude that $\Gamma^h(z)$ is strongly thin in the connected $K$-variety $Y_z$. From $z\notin \Delta$ we conclude that 
$\Gamma^v(z)=\emptyset$. Hence $\Gamma(z)$ is strongly thin in $Y_z$. Furthermore the open subscheme $V_z$ of $Y_z$ is 
not empty because $z\in W\subset f(V)$. Finally we have $z\in \Xi$ and thus $\Sigma\cap Y_z(K)$ is not strongly thin in $Y_z$. Hence we can choose a point $y\in (\Sigma\cap V_z(K))\setminus \Gamma(z)$. That point $y$ lies in $(\Sigma\cap V(K))\setminus \left( \bigcup_{i\in I} \Gamma_i\right)$. 
\end{proof}



As a by-product we can quickly establish the following two product theorems over Hermite-Minkowski schemes $S$. They are already known in the case where $S$ is of finite type over $\Spec(\Zz)$ (cf. \cite[Thm. 1.4]{Lug3}, \cite[Thm. 1.9]{CDJLZ}), and one could prove them equally well, proceeding along the lines of the original proofs in \cite{Lug3} and \cite{CDJLZ} to establish them. We include them because we think they might be of use in future work, when studying Hilbertianity questions over fields $K$ of characteristic zero that are not finitely generated, as for example in \cite{Jav2}.

\begin{coro} \label{coro:pt1} Let $S$ be a Hermite-Minkowski scheme and $K=R(S)$. Let $\ZZ_1$ and $\ZZ_2$ be $S$-varieties. If $\ZZ_1$ and $\ZZ_2$ both have WHP integrally over $S$, then $\ZZ_1\times_S \ZZ_2$ has WHP integrally over $S$. 
\end{coro}

\begin{proof} Let $\YY=\ZZ_1\times_S \ZZ_2, \ZZ=\ZZ_2$ and $F: \YY\to \ZZ$ the projection of the fibre product. Let $Y=\YY_K, Z=\ZZ_K, Z_j=\ZZ_{j,K}$ and $f=F_K$. Then $Y=Z_1\times Z_2$ and $f: Y\to Z$ is the projection. This is a flat morphism. The variety $Z_j$ is normal connected $Z_j(K)\neq \emptyset$. Hence $Z_j$ is geometrically connected. For every $z\in Z$ the fibre $f^{-1}(z)$ is isomorphic to $Z_{1, k(z)}$. Thus the morphism $f$ is normal and all its fibres are geometrically connected. Let $\Xi_0$ be the set of all $z\in\ZZ(S)^{(1)}$ such that 
$\YY_z$ has WHP integrally over $S$. The data are as in Setup \ref{setup2}. Now $\YY_z=\ZZ_1\times_S D(z)$ and $D(z)$ is a non-empty open subscheme of $S$ that contains all codimension $1$ points of $S$. Thus $\YY_z$ has WHP integrally over $S$ for all $z\in \ZZ(S)^{(1)}$. It follows that $\Xi_0=\ZZ(S)^{(1)}$. Furthermore $\ZZ$ has WHP integrally over $S$. The morphism $f$ satisfies condition (EP) by Lemma \ref{lemm:EPprod}. Thus Theorem \ref{thm:iwhp} implies the assertion. 
\end{proof}

\begin{coro} \label{coro:pt2} Let $S$ be a Hermite-Minkowski scheme and $K=R(S)$. Let $Z_1$ and $Z_2$ be normal proper geometrically connected 
$K$-varieties. If $Z_1$ and $Z_2$ both have WHP over $K$, then $Z_1\times_S Z_2$ has WHP over $K$. 
\end{coro}

\begin{proof} There exists a non-empty open subscheme $S'$ of $S$ so that $Z_1$ and $Z_2$ extend to proper $S$-varieties $\ZZ_1$ and $\ZZ_2$. The scheme $S'$ is a Hermite-Minkowski scheme. We have $\ZZ_j(S)^{(1)}=Z_j(K)$ by the properness. Thus $\ZZ_j$ has WHP integrally over $S$ for $j\in\{1,2\}$. By Corollary 
\ref{coro:pt1} it follows that $\ZZ_1\times_S \ZZ_2$ has WHP integrally over $S$. This implies that $Z_1\times_K Z_2$ has WHP over $K$.
\end{proof}

The following product theorem slightly generalizes \cite[Theorem 2.5.(1)]{Jav2}.

\begin{coro} \label{coro:pt} Let $k$ be field of characteristic zero. Assume that $\Gal(k)$ is topologically finitely generated (e.g. $k\in\{\Cc,\Rr,\Qq_p\}$). Let $K/k$ be a finitely generated transcendental field extension. Let $Z_1, Z_2$ be normal proper geometrically connected $K$-varieties. If $Z_1$ and $Z_2$ both have WHP over $K$, then $Z_1\times_K Z_2$ has WHP over $K$. 
\end{coro}

\begin{proof} There exists a smooth connected $k$-variety $S$ with function field $K$. By Proposition \ref{prop:HM} the scheme $S$ is a Hermite-Minkowski scheme. Thus the assertion follows with the help of Corollary \ref{coro:pt2}. 
\end{proof}

\section{Digression on \'etale homotopy theory}\label{sec:digression}
Making use of the higher \'etale homotopy theory of Artin and Mazur \cite{AM} and the long exact sequence of homotopy groups of Friedlander \cite{F1} we can weaken the properness assumption in Theorem \ref{thm:1} and Theorem \ref{thm:2} to some extent. 
The following definition stems from Friedlander's paper \cite{F1}. 

\begin{defi} \label{defi:fib}  Let $Y, Z$ be normal connected varieties over a field $K$.
Let $f: Y\to Z$ be a $K$-morphism. 
We say that $f$ is a {\bf special fibration} if $f$ is surjective and there exist a $Z$-scheme $\oY$, smooth and proper over $Z$, and 
irreducible closed subschemes $T_1,\cdots, T_n$ of $\oY$, each of pure relative codimension, with the property that every intersection
$T_{i_1}\cap \cdots \cap T_{i_k}$ is smooth over $Z$, and an $Z$-isomorphism $Y\to \oY\setminus \bigcup_{j=1}^n T_j$.
We say that $f$ is a {\bf geometric fibration} if there exists for every $z\in Z$ a non-empty open subscheme $U$ of $Z$ such that $z\in U$ and such that
$f_U: Y_U\to U$ is a special fibration. 
\end{defi}

Note that every geometric fibration is smooth and surjective. On the other hand, if $f: Y\to Z$ is smooth and proper, then $f$ is a geometric fibration.

\begin{defi} Let $Z$ be a connected variety over a field $K$ of characteristic zero. One calls $Z$ a {\bf $K(\pi, 1)$-scheme} if 
for every locally constant constructible abelian sheaf $F$ on $Z$ and some (hence every) geometric point $\oz$ of $Z$ the 
canonical maps 
$$H^q(\pi_1(Z, \oz), F_\oz)\to  H^q_\mathrm{et}(X, F)$$
are isomorphisms for all $q\ge 0$. 
\end{defi}

See \cite[Definition 4.1.1, Proposition 4.1.2]{aich} for various equivalent characterizations. 
I learnt the proof of the following theorem from Jakob Stix. 

\begin{thm}  \label{thm:stix}
Let $Y$ and $Z$ be connected normal varieties over a field $K$ of characteristic zero. Let $f: Y\to Z$ be a geometric fibration. Assume that all fibres of $f$ are geometrically connected. 
\begin{enumerate}
\item[(a)] If $Z$ is a $K(\pi, 1)$-scheme, then $f$ satisfies condition (EP). 
\item[(b)] If $f$ has a section, then $f$ satisfies condition (EP). 
\end{enumerate}
\end{thm}

\begin{proof} Let $\oxi$ be a geometric generic point of $Z$ and $a\in Y_\oxi(k(\oxi))$. Let $i: Y_\oxi=Y\times_Z \Spec(k(\oxi))\to Y$ be the 
projection of the fibre product. 
There is the 
long exact  homotopy sequence of Friedlander 
(cf. \cite[Theorem 4.4]{F1}). 
If $Z$ is a $K(\pi,1)$-scheme, then $\pi_2(Z, \oxi)$ is trivial (cf. \cite[Proposition 4.1.4]{aich}), and thus $\pi_1(Y_\oxi, \oa)\to \pi_1(Y, i(\oa))$ is injective. If $f$ has a section, then 
this induces a section of the morphism $\pi_2(Y, i(a))\to \pi_2(Z, \oxi)$ of pro-groups, and thus
$\pi_1(Y_\oxi, \oa)\to \pi_1(Y, i(\oa))$ must be injective also in that case. This finishes up the proof of (a) and (b). 
\end{proof}



\begin{coro} \label{coro:stix} 
Let $Y$ and $Z$ be connected smooth varieties over a field $K$ of characteristic zero. Let $f: Y\to Z$ be a geometric fibration. 
If $Y(R(Z))^{(1)}\neq \emptyset$, then all fibres of $f$ are geometrically connected and $f$ satisfies condition (EP). 
\end{coro}

\begin{proof} There exists a smooth proper morphism $\of:\oY\to Z$ and 
and an open immersion $Y\to \oY$ over $Z$. Replacing $\oY$ by that connected component of $\oY$ that contains $Y$ we can assume that $\oY$ is connected. Thus $\oY$ is a smooth connected $K$-variety. Furthermore $\oY(R(Z))\neq \emptyset$. By Corollary \ref{coro:SGA1ht} all fibres of $\of$ are geometrically connected. Moreover the fibres of $\of$ are smooth. It follows that the fibres of $\of$ are geometrically integral. The morphism $f$ is a geometric fibration, hence surjective. Thus, for every $z\in Z$, the fibre $Y_z$ is a non-empty open subscheme of the geometrically integral $k(z)$-variety $\oY_z$. Hence all fibres of $f$ are geometrically connected. 
There exists a non-empty open subscheme $D$ of $Z$ such that $D$ contains all codimension $1$ points of $Z$ and such that the geometric fibration
$f_D: Y_D\to D$ has a section (cf. Proposition \ref{lemm:vojjav}). It follows by Theorem \ref{thm:stix} that $f_D$ satisfies condition (EP). 
Lemma \ref{lemm:ht-loc} implies that $f$ satisfies condition (EP). 
\end{proof}

\begin{coro} \label{coro:pi2} Let $K$ be a field of characteristic zero. Let $Y$ and $Z$ be connected smooth $K$-varieties. 
Let $f: Y\to Z$ be a geometric fibration. 
Assume that $Z$ has HP over $K$ and that the subset $Y(Z)^{(1)}$ of $Y(R(Z))$ is not strongly thin in $Y_{R(Z)}$. Then 
$Y$ has WHP over $K$. 
\end{coro}

\begin{proof}
Corollary \ref{coro:stix} implies that $f$ satisfies condition (EP). Thus the assertion follows by Theorem \ref{thm:main:abstract}. 
\end{proof}

\begin{coro} \label{coro:iwhp-setb} Consider the situation of Setup \ref{setup2}. Assume that $Y$ and $Z$ are smooth over $K$ and that 
$f$ is a geometric fibration. Assume that all fibres of $f$ are geometrically connected. Assume that $f$ has a section or that $Z$ is a $K(\pi,1)$-scheme. 
If $\Xi$ is not strongly thin in $Z$, then $\Sigma$ is not strongly thin in $Y$. In particular, if $\Xi_0=\{z\in \ZZ(S)^{(1)}: \mbox{$\YY_z$
has WHP integrally over $S$}\}$ is not strongly thin in $Z$, then $\YY$ has WHP integrally over $S$. 
\end{coro}

\begin{proof} By Theorem \ref{thm:stix} $f$ satisfies condition (EP). Thus the assertions follow by Theorem \ref{thm:iwhp}. 
\end{proof}

\section{Discussion of an open problem}\label{sec:open}

The following conjecture is in my opinion the most innocent looking special case of Conjecture \ref{conj:ell2} that is unsolved at present, and I want to explain briefly where the problem lies. 

\begin{conj} Let $K$ be a number field and $Z\subset \Pp^1_K$ an open subscheme. Let $F=R(Z)$. Let $E_F$ be an elliptic curve over 
$F$ or rank $\dim_\Qq(E(F)\otimes_\Zz\Qq)\ge 1$ and with semistable reduction along $Z$. Let $f: E\to Z$ be the N\'eron model of $E_F$. Then 
$E$ has WHP over $K$. 
\end{conj}

\begin{rema} Let $S$ be the set of all points $z\in Z$ such that $E$ has bad reduction at $z$ and $U=Z\setminus S$. Then $E_U\to U$ is an abelian scheme and Corollary \ref{coro:as} implies that $E_U$ has WHP over $K$. Nevertheless the conjecture stays open unless $S=\emptyset$. 
If $S\neq \emptyset$, then the above results cannot be applied to $f$ because $f$ then does not satisfy condition (EP)
by the following Proposition. 
\end{rema}

\begin{prop}\label{prop:htgr} If $f$ satisfies condition (EP), then $S=\emptyset$. 
\end{prop}

\begin{proof} Let $\ol{E}/Z$ be the minimal regular proper model of $E_F$. The morphism $\of: \oE\to Z$ is flat and its fibres are (geometrically) reduced and geometrically connected. For every $z\in S$ the singular locus $\mathrm{Sing}(\oE_z)$ is a finite set because $\oE_z$ is reduced. 
The N\'eron model $E$ is obtained from $\oE$ by removing $\bigcup_{z\in S} \mathrm{Sing}(\oE_z)$ from $\oE$. Hence $E$ is an open subscheme of $\oE$ that contains all codimension $1$ points of $\oE$ (*). 

Let $\oz$ be a geometric point of $Z$ 
and  $\oa$ be a geometric point of $E_\oz$. 
There is an exact sequence
$\pi_1(\oE_\oz, \oa)\to \pi_1(\oE, \oa)\to \pi_1(Z, \oz)\to 1$. 
The homomorphism 
$\pi_1(E, \oa)\to \pi_1(\oE, \oa)$ is an isomorphism due to (*) by purity of the branch locus. It follows that
$$\pi_1(\oE_\oz, \oa)\to \pi_1(E, \oa)\to \pi_1(Z, \oz)\to 1$$
is exact. Let $\oxi$ be a geometric generic point of $Z$ and $\ob$ a geometric point of $E_\oxi$. Choose an \'etale path $c\in \pi_1(E, \ob, \oa)$ and let $c'\in \pi_1(Z, \oxi, \oz)$ be the \'etale path induced by $c$. We have a commutative diagram with exact rows
$$\xymatrix{ 1\ar[r] &\pi_1(E_\oxi, \ob)\ar[r]& \pi_1(E, \ob)\ar[r]\ar[d]^{c(-)c^{-1}} &\pi_1(Z, \oxi)\ar[r]\ar[d]^{c'(-)c'^{-1}} &1\\
& \pi_1(\oE_\oz, \oa)\ar[r]& \pi_1(E, \oa)\ar[r]& \pi_1(Z, \oz)\ar[r]& 1}$$
where the vertical maps are the isomorphisms. It follows that there is an epimorphism 
$$\pi_1(\oE_\oz, \oa)\to \ker( \pi_1(E_\oxi, \ob)\to \pi_1(Z,\oxi))\cong \pi_1(E_\oxi, b)$$
for every geometric point $\oz$ of $Z$ and every geometric point $\oa$ of $\oz$ (*). 
It is known that $\pi_1(E_\oxi, \ob)\cong\widehat{\Zz}^2$ ($E_\oxi$ is an elliptic curve over an algebraically closed field of characteristic zero). 
 We want to prove that $S\neq \emptyset$. Assume to the contrary that $S$ is not empty and choose in the above $\oz$ localized in $S$.
 Then the reduction in $\oz$ is bad semistable. Hence $\oE_\oz$ is a N\'eron $n$-gon and it is known that $\pi_1(\oE_\oz,\oa)=\widehat{\Zz}$, which leads by (*)  to an epimorphism $\hat{\Zz}\to \hat{\Zz}^2$. Contradiction. 
\end{proof}

{\sc Sebastian Petersen\\ 
Universit\"at Kassel\\
Fachbereich 10\\
Wilhelmsh\"oher Allee 71--73\\
34121 Kassel, Germany}\\
E-mail address: \texttt{petersen@mathematik.uni-kassel.de}

\end{document}